\documentclass[reqno]{amsart}
\usepackage{a4wide}
\usepackage{hyperref}
\usepackage{xcolor}
\usepackage{amsmath}
\usepackage{amssymb}
\usepackage{amsthm}
\hypersetup{ colorlinks = true, urlcolor = blue, linkcolor = blue, citecolor = red }
\DeclareMathOperator{\dv}{div}
\DeclareMathOperator{\loc}{loc}

\newcommand{\RR}{\mathbb{R}}
\newcommand{\mA}{\mathcal{A}}
\newcommand{\Om}{\Omega}
\newcommand{\na}{\nabla}

\newcommand{\La}{\Lambda}
\newcommand{\al}{\alpha}

\newcommand{\ga}{\gamma}
\newcommand{\de}{\delta}

\newcommand{\la}{\lambda}

\newcommand{\iprod}[2]{\langle #1,  #2\rangle}
\usepackage{stmaryrd}

\DeclareRobustCommand{\rchi}{{\mathpalette\irchi\relax}}
\newcommand{\irchi}[2]{\raisebox{\depth}{$#1\chi$}} 

\newcommand{\norm}[1]{\left\lVert#1\right\rVert}
\usepackage{textcomp}
\makeatletter
\@namedef{subjclassname@2020}{\textup{2020} Mathematics Subject Classification}
\makeatother
\newtheorem{theorem}{Theorem}[section]
\newtheorem{lemma}[theorem]{Lemma}

\newtheorem{definition}[theorem]{Definition}

\newtheorem{remark}[theorem]{Remark}

\def\Xint#1{\mathchoice
	{\XXint\displaystyle\textstyle{#1}}%
	{\XXint\textstyle\scriptstyle{#1}}%
	{\XXint\scriptstyle\scriptscriptstyle{#1}}%
	{\XXint\scriptstyle\scriptscriptstyle{#1}}%
	\!\int}
\def\XXint#1#2#3{{\setbox0=\hbox{$#1{#2#3}{\int}$}
		\vcenter{\hbox{$#2#3$}}\kern-.5\wd0}}

\def\YYint#1#2#3{{\setbox0=\hbox{$#1{#2#3}{\iint}$}
		\vcenter{\hbox{$#2#3$}}\kern-.51\wd0}}
 
\def\Xint#1{\mathchoice
	{\XXint\displaystyle\textstyle{#1}}%
	{\XXint\textstyle\scriptstyle{#1}}%
	{\XXint\scriptstyle\scriptscriptstyle{#1}}%
	{\XXint\scriptscriptstyle\scriptscriptstyle{#1}}%
	\!\int}
\def\XXint#1#2#3{{\setbox0=\hbox{$#1{#2#3}{\int}$ }
		\vcenter{\hbox{$#2#3$ }}\kern-.6\wd0}}

\def\dashint{\Xint-}
\newcommand{\fint}{\dashint}

\DeclareMathOperator{\dist}{dist}
\usepackage{nameref}
\makeatletter
\let\orgdescriptionlabel\descriptionlabel
\renewcommand*{\descriptionlabel}[1]{%
	\let\orglabel\label
	\let\label\@gobble
	\phantomsection
	\edef\@currentlabel{#1}%
	\let\label\orglabel
	\orgdescriptionlabel{#1}%
}
\makeatother
\numberwithin{equation}{section}

\begin{document}

\title{Self-improving properties of very weak solutions to double phase systems}

\everymath{\displaystyle}

\author{Sumiya Baasandorj}
\address[Sumiya Baasandorj]{Department of Mathematical Sciences, Seoul National University, Seoul 08826, Korea.}
\email[Corresponding author]{summa2017@snu.ac.kr}
\thanks{S. Baasandorj and W. Kim were supported by the National Research Foundation of Korea NRF-2021R1A4A1027378.}

\author{Sun-Sig Byun}
\address[Sun-Sig Byun]{Department of Mathematical Sciences, Seoul National University, Seoul 08826, Korea}
\address{Research Institute of Mathematics, Seoul National University, Seoul 08826, Korea.}
\email{byun@snu.ac.kr}
\thanks{S. Byun was supported by the National Research Foundation of Korea NRF-2022R1A2C1009312.}

\author{Wontae Kim}
\address[Wontae Kim]{Department of Mathematics, Aalto University, P.O. BOX 11100, 00076 Aalto, Finland}
\email{wontae.kim@aalto.fi}

\begin{abstract}
	We prove the self-improving property of very weak solutions to non-uniformly elliptic problems of double phase type in divergence form under sharp assumptions on the nonlinearity.
\end{abstract}

\keywords{Very weak solution, Double phase problems, Non-standard growth, Lipschitz truncation and Whitney covering, Gehring lemma}
\subjclass[2020]{35D30, 35J60, 35J70}
\maketitle

\section{Introduction}

In this paper, we consider self-improving properties of very weak solutions to the non-uniformly elliptic problems of double phase type in the divergence of the following form
\begin{align}
	\label{me}
	-\dv\mA(x,\na u)
	= 
	-\dv\left(|F|^{p-2}F + a(x)|F|^{q-2}F \right)
	\quad\text{in}\quad \Omega
\end{align}  
for a bounded domain $\Omega\subset \RR^{n}$ with $n\geqslant 2$, where the map $\mA(x,\xi):\Omega\times\RR^{Nn}\longrightarrow \RR^{Nn}$ $(N\geqslant 1)$ is a Carath\'eodory vector field satisfying the following structure assumptions with fixed constants $0<\nu\le L<\infty$: 

\begin{equation}
	\label{str}
	\begin{cases}
		\nu(|\xi|^p+a(x)|\xi|^q)\leqslant \iprod{\mA(x,\xi)}{\xi}, \\
		|\mA(x,\xi)|\leqslant L(|\xi|^{p-1}+a(x)|\xi|^{q-1})
	\end{cases}
\end{equation}
for a.e. $x\in\Omega$ and every $\xi\in\RR^{Nn}$, whereas $F : \Omega \rightarrow \RR^{Nn}$  is a given vector field. Throughout the paper, we shall assume that exponents $1<p<q<\infty$ and the coefficient function $a : \Omega\rightarrow \RR$ satisfy the following main assumptions

\begin{align}
	\label{exp}
	\frac{q}{p} < 1+\frac{\alpha}{n},\quad
	0\leqslant a(\cdot)\in C^{0,\alpha}(\Omega)
	\quad\text{ for some }\quad
	\alpha\in (0,1].
\end{align}

The primary model system \eqref{me} originates from the functional given by 

\begin{align*}
	W^{1,1}(\Omega,\RR^{N})\ni v\mapsto  \mathcal{P}(v,\Omega) - \int_{\Omega}\iprod{|F|^{p-2}F + a(x)|F|^{q-2}F}{Dv}\,dx,
\end{align*}
where the double phase functional is of the form

\begin{align}
	\label{dbfunc}
	W^{1,1}(\Omega,\RR^{N})\ni v\mapsto \mathcal{P}(v,\Omega):= \int_{\Omega}H(x,|\nabla v|) \,dx.
\end{align}
Here and in the rest of the paper we denote 
\begin{align}
	\label{H-func}
	H(x,z):= |z|^{p} + a(x)|z|^{q}.
\end{align}
The function $H(x,z)$, with some ambiguity of notation, will be considered in all cases $z\in\RR$, $z\in\RR^{N}$ and $z\in\RR^{Nn}$. The double phase functional was introduced first by Zhikov \cite{Zh1,ZKO} in order to produce models of strongly anisotropic materials in the settings of homogenization and nonlinear elasticity. The main feature of the functional $\mathcal{P}$ in \eqref{dbfunc} is that its growth and ellipticity ratio depend on the modulating coefficient function $a(\cdot)$, which exhibits the mixture of two different materials. The functional in \eqref{dbfunc} itself is a significant example of functionals belonging to a family of functionals having non-standard $(p,q)$-growth conditions, as it was introduced first by Marcellini in \cite{Ma2,Ma3}. Recently, regularity properties of weak solutions to the double phase problems have been extensively investigated in a series of papers \cite{BCM1,BCM3,CM1,CM2,CM3,DM1}. Among them, the gradient H\"older regularity of a minimizer of the functional $\mathcal{P}$ has been proved under the assumption \eqref{exp} in \cite{BCM1,BCM3}. Essentially, the condition \eqref{exp} is sharp in the sense of the regularity theory considered there, see for instance \cite{ELM1,ELM2,ELM3}. Our purpose in this paper is to consider very weak solutions to the system of double phase type of equations \eqref{me}, whose definition is given as follows.

\begin{definition}
	\label{def:vw}
	For a given vector field $F : \Omega\rightarrow \RR^{Nn}$ such that 
	\begin{equation}
		\label{def:1}
		\int_{\Om}\left(|F|^{p-1}+a(x)|F|^{q-1}\right)\ dx<\infty,
	\end{equation}
	a map $u\in W^{1,1}(\Omega,\RR^N)$ with
	\begin{equation}
		\label{def:2}
		\int_{\Om}\left(|\na u|^{p-1}+a(x)|\na u|^{q-1}\right)\ dx<\infty
	\end{equation}
	is called a very weak solution to the system \eqref{me} under the assumptions \eqref{str}
	and \eqref{exp} if
	\begin{equation}
		\label{def:3}
		\int_{\Om}\iprod{\mA(x,\na u)}{\na \varphi}\ dx=\int_{\Omega}\iprod{|F|^{p-2}F+a(z)|F|^{q-2}F}{\na \varphi}\ dx
	\end{equation}
	holds for every $\varphi\in C_0^\infty(\Omega,\RR^N)$.
\end{definition}

Note that if we replace the assumptions \eqref{def:1} and \eqref{def:2}  with 
\begin{align}
	\label{vw:1}
	\int_{\Omega} H(x,|F|) \,dx < \infty
\end{align}
and
\begin{align}
	\label{vw:2}
	\int_{\Om}H(x,|\nabla u|)\ dx<\infty
\end{align}
in the above definition, respectively, where the function $H$ is defined in \eqref{H-func}, then this very weak solution $u$ of \eqref{me} is called a classical weak solution naturally. Moreover, it has been shown that, for a weak solution $u$ to \eqref{me}, the equality \eqref{def:3} under the assumption \eqref{exp} still holds for any function $\varphi \in W^{1,1}_{0}(\Omega,\mathbb{R}^N)$ with $H(x,|\nabla \varphi|)\in L^{1}(\Omega)$, see \cite{CM3}. In fact, the required integrability of $\na u$ in \eqref{vw:2} is needed in order to apply energy estimates for deriving existence, uniqueness and further regularity results. But by weakening the integrability of $\na u$ as in Definition \ref{def:vw}, we are still able to consider so-called a very weak solution $u$ to the elliptic problem \eqref{me}. At this stage, we are not allowed to use energy estimate methods due to the lack of integrability of $\na u$. In this paper, we are interested in validity of the self-improving property of very weak solutions to the system \eqref{me} modeled on the operator of double phase structure, that is, if $\left[H(x,|\nabla u|)\right]^{\delta}\in L^{1}(\Omega)$ holds for some constant $\delta\in (0,1)$ being enough to close to 1 and the non-homogeneous term $F$ has a certain integrability as in \eqref{vw:1}, then this very weak solution $u$ becomes a weak solution.

In the case of $p$-Laplace ($a(\cdot)\equiv 0$ in \eqref{me}), the self-improving property of very weak solutions have been achieved in \cite{IS} for the homogeneous case ($F\equiv 0$). In \cite{L}, this result has been extended to the $p$-Laplace type system under more general structure assumptions involving non-homogeneous data $F\in L^{p}$ based on techniques of Lipschitz truncation, which provides a way for constructing admissible test functions by applying the Whitney covering lemma. Moreover, the techniques of Lipschitz truncation have been extended and applied to Calder\'on-Zygmund type estimates \cite{AP2,AP1}, the existence of a very weak solution \cite{BDS,BS,MS} and the setting of Orlicz space \cite{BL}, the variable exponent $p(\cdot)$ space \cite{BZ}, as well as parabolic $p$-Laplace type systems \cite{AB,B,BBS,KL}. See also  \cite{AF1,AF2} for the application of Lipschitz truncation method on elliptic problems.

Going back to the double phase system with the non-homogeneous data, we aim at proving the self-improving property for the double phase problems by revisiting Lipschitz truncation techniques in \cite{L,KL}. The main theorem in this paper reads as follows:

\begin{theorem}\label{main}
	Let $F : \Omega\rightarrow \RR^{Nn}$ be a vector field satisfying
	\begin{align*}
		\int_{\Om}H(x,|F|)\ dx<\infty
	\end{align*}
	under the assumptions \eqref{str} and \eqref{exp}. 
	There exists an exponent $\delta_{0}\in (1-1/q,1)$ depending only on $n,N,p,q,\alpha,\nu, L$ and $[a]_{0,\alpha}$ such that for every $\delta\in (\delta_0,1)$ and every very weak solution $u$ of \eqref{me} with 
	\begin{align}
        \label{main:1}
		\int_{\Omega}\left[ H(x,|\nabla u|) \right]^{\delta}\,dx < \infty,
	\end{align}
	there exists a positive radius $R_0$ depending only on $n,N,p,q,\alpha,\nu,L,[a]_{0,\alpha},\delta$ and $ \norm{\na u}_{L^{p\delta}(\Omega)}$ such that
	\begin{align*}
		\begin{split}
			\fint_{B_{r}(x_0)}H(x,|\nabla u|)\ dx
			&\leqslant c\left(\fint_{B_{2r}(x_0)}\left[H(x,|\nabla u|)\right]^\de \ dx\right)^\frac{1}{\de}
			+c \fint_{B_{2r}(x_0)}H(x,|F|)\ dx + c\\
		\end{split}
	\end{align*}	
	holds for some constant $c\equiv c(n,N,p,q,\alpha,\nu, L, [a]_{0,\alpha})$, whenever $B_{2r}(x_0)\subset\Om$ is a ball with $2r\leqslant R_0$.
\end{theorem}

The result of the above theorem is new in the sense of the self-improving property of very weak solutions to double phase problems as far as we are concerned. Clearly, the result of the above theorem can be reduced to the standard $p$-Laplace case \cite{IS,L} when $a(\cdot)\equiv 0$.  Let us neatly explain the key ideas of the proof of Theorem \ref{main}. The first target is to obtain a Caccioppoli type inequality of Lemma \ref{caccio} via the techniques of Lipschitz truncation. With a fixed ball $B_{2R}\subset\Omega$ of a suitable size $R$ and a positive number $\lambda$ appropriately large, we construct an admissible function $\phi_{\lambda}\in W^{1,\infty}_{0}(B_{2R})$ satisfying 
\begin{align}
	\label{at:1}
	\frac{|\phi_{\lambda}(x)|}{R}+|\na \phi_{\lambda}(x)|\lesssim\la^\frac{1}{p}\text{ in }B_{R}
\end{align}
and 
\begin{align}
	\label{at:2}
	[a(x)]^\frac{1}{q}\frac{|\phi_{\la}(x)|}{R}+[a(x)]^\frac{1}{q}|\na\phi_{\lambda}(x)|\lesssim \la^\frac{1}{q}\text{ in }B_{R}
\end{align}
(see Lemma \ref{lip1} and Lemma \ref{lip4} below) in order to estimate the key term
\begin{align*}
	I:=\int_{\left\{x\in B_{R} :\, H(x,|\na u|)>\la\right\}}\left(|\na u|^{p-1}+a(x)|\na u|^{q-1}\right)|\na \phi_{\lambda}|\ dx
\end{align*}
in a proper way (see Lemma \ref{pre_p_reverse}). Using only \eqref{at:1}, it can be seen
\begin{align}
	\label{at:4}
	\begin{split}
		I\lesssim&\int_{\left\{x\in B_{R} :\, H(x,|\na u|)>\la\right\}}|\na u|^{p-1}\lambda^\frac{1}{p}\ dx\\
		&+\int_{\left\{x\in B_{R} :\, H(x,|\na u|)>\la\right\}}a(x)|\na u|^{q-1}\lambda^\frac{1}{p}\ dx.
	\end{split}
\end{align}
However, it is possible to deal with the first term of \eqref{at:4} as in the $p$-Laplacian case, but the second term in \eqref{at:4} causes a trouble to make further estimates. To overcome such a difficulty, we apply \eqref{at:2} to estimate \eqref{at:4} in the following way

\begin{align*}
	\begin{split}
		I\lesssim&\int_{\left\{x\in B_{R} :\, H(x,|\na u|)>\la\right\}}|\na u|^{p-1}\la^\frac{1}{p}\ dx\\
		&+\int_{\left\{x\in B_{R} :\, H(x,|\na u|)>\la\right\}}[a(x)]^\frac{q-1}{q}|\na u|^{q-1}\lambda^\frac{1}{q}\ dx.
	\end{split}
\end{align*}
Then the second term of the above display can be treated as usually done for the $q$-Laplacian case.

\begin{remark}
	We remark that $\delta_0\in(1-1/q,1)$ in the statement of Theorem \ref{main} is close enough to $1$ due to a counter example constructed for the $p$-Laplace equation $(a(\cdot)\equiv 0)$, see \cite{CT}. We also would like to point out that the main assumption \eqref{exp} is optimal in the sense that, for a given very weak solution $u\in W^{1,1}(\Omega,\RR^{N})$ to the system \eqref{me} with 
	\begin{align}
		\label{rmk:1}
		\int_{\Omega}\left[H(x,|\na u|)\right]^{\delta}\,dx < \infty
	\end{align}
for any exponent $\delta\in (0,1)$ sufficiently close to $1$, we would have 
\begin{align}
	\label{rmk:2}
		\int_{\Omega}\left[H(x,|\na u|)\right]^{\delta}\,dx \approx
		\int_{\Omega} \left[|\na u|^{p\delta} + a_{\delta}(x)|\na u|^{q\delta}\right]\,dx=: \mathcal{P}_{\delta}(u,\Omega),
\end{align}
where the coefficient function defined by $a_{\delta}(x):= [a(x)]^{\delta}$ is a member of $C^{0,\alpha\delta}(\Omega)$, and we need the following condition for the new functional $\mathcal{P}_{\delta}$ in \eqref{rmk:2}
\begin{align*}
	\delta q\leqslant \delta p+\frac{(\alpha \delta)(\delta p)}{n}
\end{align*}
in order to have the absence of Lavrentiev phenomenon, see for instance \cite[Theorem 4.1]{CM1} and \cite{Zh1, ZKO} for details. In turn, the last display yields that
\begin{align*}
	q\leqslant p+\frac{\alpha \delta p}{n} <p+\frac{\alpha p}{n}
\end{align*}
for every $\delta\in (0,1)$ enough close to $1$. In this regard, the delicate borderline case $q = p + \frac{p\alpha}{n}$ is not considered in this paper.
\end{remark}

\begin{remark}
Here we point out that if the display \eqref{main:1} is valid for some $\delta\in (0,1)$ sufficiently close to 1, then \eqref{def:2} holds true by means of H\"older inequalities that
\begin{align*}
\begin{split}
    &\int_{\Omega}\left(|\na u|^{p-1}+a(x)|\na u|^{q-1}\right)\ dx\\
    &\le |\Om|^\frac{\delta p-p+1}{\delta p}\left(\int_{\Omega}|\na u|^{\delta p}\ dx\right)^\frac{p-1}{\delta p}\\
    &\quad+\left(\int_{\Om}[a(x)]^\frac{\delta}{\delta q-q+1}\ dx\right)^\frac{\delta q-q+1}{\delta q}\left(\int_{\Om}[a(x)|\na u|^{q}]^\delta\ dx\right)^\frac{q-1}{\delta q}\\
    &\le \Bigl(|\Om|^\frac{\delta p-p+1}{\delta p}+\|a\|_{L^\infty(\Om)}^\frac{1}{q}|\Om|^\frac{\delta q-q+1}{\delta q}\Bigr)\left(\int_{\Om}[H(x,|\na u|)]^\delta\ dx+1\right)^\frac{q-1}{\delta q},
\end{split}
\end{align*}
where we have also used the simple identity  $a(x)=[a(x)]^\frac{1}{q}[a(x)]^\frac{q-1}{q}$.
\end{remark}

Finally, we outline the organization of the paper. In the next section, we introduce the notations and basic tools to be used throughout the paper. Section \ref{sec3} is devoted to obtaining the reverse H\"older inequality of a very weak solution to the system \eqref{me}. In the last section, we provide the proof of the main Theorem \ref{main}.

\section{Preliminaries}
We denote by $B_{r}(x_0) = \{ x\in\RR^n : |x-x_0|<r \}$ the open ball in $\RR^n$ centered at $x_0\in \RR^n$ with radius $r> 0$. If the center is clear in the context, we shall denote it by $B_{r}\equiv
B_{r}(x_0)$. As usual, we denote by $c$ a generic positive constant, possibly varying from line to line; special constants will be denoted by symbols such as $c_1, c_{*}, c_{\varepsilon}$, and so on. All such constants will always be larger than one; moreover, relevant dependencies on parameters will be emphasized using brackets, that is, for example, $c\equiv c(n,p,q,\nu,L)$ means that $c$ is a constant depending only on $n,p,q,\nu,L$. 
With $f:\mathcal{B}\rightarrow \RR^{k}$ $(k\geqslant 1)$ being a measurable map for a measurable subset $\mathcal{B}\subset\RR^n$ having a finite and positive measure, we denote by
\begin{align*}
	(f)_{\mathcal{B}}\equiv \fint_{\mathcal{B}} f(x)\,dx = \frac{1}{|\mathcal{B}|}\int_{\mathcal{B}}f(x)\,dx
\end{align*}
to mean its integral average over $\mathcal{B}$. For a given number $\alpha\in (0,1]$, the space $C^{0,\alpha}(\Omega)$ consists of measurable functions $f:\Omega\rightarrow \RR$ such that 
\begin{align*}
	\norm{f}_{C^{0,\alpha}(\Omega)} := [f]_{0,\alpha;\Omega} + \norm{f}_{L^{\infty}(\Omega)}<\infty,
\end{align*}
where for any open subset $\mathcal{B}\subset\Omega$, the semi-norm is defined by 
\begin{align*}
	[f]_{0,\alpha;\mathcal{B}}:= \sup\limits_{x,y\in\mathcal{B},x\neq y} \frac{|f(x)-f(y)|}{|x-y|^{\alpha}}
	\quad\text{and}\quad
	[f]_{0,\alpha}\equiv [f]_{0,\alpha;\Omega}.
\end{align*}

Next, we introduce some auxiliary tools to be used in the later parts of the paper. 
For a given function $f\in L^1_{\loc}(\RR^n)$, the uncentered Hardy-Littlewood maximal function of $f$ is defined as
\begin{equation}\label{def_max}
	M(f)(x):=\sup_{x\in B_{r}(x_0)}\fint_{B_r(x_0)}|f(y)|\ dy,
\end{equation}
where supremum is taken over all balls containing the point $x$. The first tool we report here is the following classical strong type estimate of the maximal operator, see for instance \cite[Lemma 2.1.6]{G}.
\begin{lemma}
	\label{p_p}
	For any $f\in L^s(\RR^n)$ with $1<s<\infty$, there exists a constant $c(n,s)\equiv 2^\frac{n}{s}s(s-1)^{-1}$ such that
	\begin{equation*}
		\int_{\RR^n}|M(f)|^s\ dx\leqslant c(n,s)\int_{\RR^n}|f|^s\ dx.
	\end{equation*}
\end{lemma}
Note that since the constant $c\equiv c(n,s)$ defined in Lemma~\ref{p_p} is continuous in $1<s<\infty$. Thus, for fixed exponents $1<s_1<s_2<\infty$, there exists a constant $c\equiv c(n,s_1,s_2)$ such that
\begin{equation*}
	\int_{\RR^n}|M(f)|^s \ dx\leqslant c\int_{\RR^n}|f|^s\ dx
\end{equation*}
holds, whenever $f\in L^s(\RR^n)$ with $1<s_1<s<s_2$. 

Next, we need a Poincar\'e type inequality in the following, see for instance \cite[Proposition 2.1]{D} and \cite[Theorem 5.47]{KLV}.
\begin{lemma}\label{boundary_poincare}
	Let $f\in W_0^{1,s}(B_{R})$ with $1<s<\infty$ and a ball $B_{R}\subset\RR^{n}$. Then there exists a constant $c\equiv c(n,s)$, which is continuous with respect to $s$-variable, such that
	\begin{equation*}
		\int_{B_{4r}(y)\cap B_{R}}|f|^s\ dx\leqslant cr^s\int_{B_{4r}(y)\cap B_{R}}|\na f|^s\ dx
	\end{equation*}
	holds, whenever $B_r(y)\subset\RR^{n}$ is a ball with $B_r(y)\cap B_{R}^{c}\ne \emptyset$.
\end{lemma}

We end up the present section with the following classical iteration lemma, which can be found in \cite[Lemma 6.1]{Gi1}.
\begin{lemma}\label{iter_lemma}
	Let $0< R_0< R_1<\infty$ be given numbers and let $h : [R_0,R_1] \to \RR$ be a non-negative and bounded function. Furthermore, let $\theta \in (0,1)$ and $A,B,\gamma_1,\gamma_2 \geq 0$ be fixed constants and	suppose that
	\[
	h(s) \leqslant \theta h(t) + \frac{A}{(s-t)^{\gamma_1}} + \frac{B}{(s-t)^{\gamma_2}}
	\]
	holds for all $R_0 \leqslant t < s \leqslant R_1$. Then there exists a positive constant $c \equiv c(\theta,\gamma_1,\gamma_2)$ such that
	$$
	h(R_0) \leqslant c \left(\frac{A}{(R_1-R_0)^{\gamma_1}} + \frac{B}{(R_1-R_0)^{\gamma_2}}\right).
	$$
\end{lemma}

\section{Reverse H\"older inequality}
\label{sec3}

Throughout this section, let $u\in W^{1,1}(\Omega,\RR^{N})$ always be a very weak solution to the system \eqref{me} in the sense of Definition \ref{def:vw} under the assumptions \eqref{str} and \eqref{exp}. A main purpose of this section is to derive a Caccioppoli type inequality and reverse H\"older type inequality for very weak solution $u$. 

\subsection{Basic settings and notations}
\label{subsec3-1} We introduce the basic settings and notations to be used throughout the present section under which we will work on.
\begin{enumerate}
	\item[1.] 	Firstly, we fix constants $\gamma, \tilde{\gamma}$ with 
\begin{equation}
	\label{def_ga}
	\frac{1}{p}<\ga:=\frac{p+2}{3p}<\tilde{\ga}:= \frac{2p+1}{3p}< 1,
\end{equation}	
	\item[2.] Next, we select a constant $\de_0\in(0,1)$ to satisfy
\begin{equation}
	\label{pre_de}
	\begin{array}{lll}
		\tilde{\gamma} <\de_0<1&\text{ and }&q< p+\frac{(\de_0 \al) p}{n},
	\end{array}
\end{equation}
where the possibility of such a choice of $\delta_0$ comes from the main assumption \eqref{exp}. The validity of $\eqref{pre_de}_{2}$ implies 
\begin{equation}
	\label{pre_a2}
	\frac{\alpha}{q}-\frac{n}{\de_0 p}\left(1-\frac{p}{q}\right)=\frac{n}{\de_0 qp}\left(\frac{\de_0 \alpha p}{n}-(q-p)\right)> 0.
\end{equation}
Note that the inequalities \eqref{pre_de}-\eqref{pre_a2} hold true with $\delta_0$ replaced by any number $\delta\in (\delta_0,1)$ once they are valid for $\delta_0\in (0,1)$ as in \eqref{pre_de}.
	\item[3.] With those constants fixed in \eqref{def_ga}-\eqref{pre_de}, let $B_{2R}\equiv B_{2R}(x_0)\subset \Omega$ be a fixed ball with $R\leqslant 1$ such that 
	\begin{equation}\label{a2}
	R^{\frac{\alpha}{q}-\frac{n}{\de p}\left(1-\frac{p}{q}\right)}\left(\int_{\Om}|\na u|^{\de p}\ dx\right)^{\frac{1}{\de p}\left(1-\frac{p}{q}\right)}\leqslant 1,
\end{equation}
\begin{equation}\label{le_la}
	\fint_{B_{2R}}\left[H(x,|\na u|)+H(x,|F|)\right]^{\de}\ dx\leqslant \La^{\de}
\end{equation}
and
\begin{equation}\label{ge_la}
	\fint_{B_{R}}\left[H(x,|\na u|)+H(x,|F|)\right]^{\de}\ dx= \La^{\de}
\end{equation}
for some fixed constants $\de\in (\de_0,1)$ and $\Lambda\geqslant 1$. 
\item[4.] With the fixed ball $B_{2R}$ satisfying \eqref{a2}-\eqref{ge_la} and the fixed constant $\gamma$ in \eqref{def_ga}, we denote
\begin{equation}
\label{def_g}
	G(x):= \left[M\left(G_p+G_q\right)(x)\right]^\frac{1}{p\gamma},
\end{equation}
where $M$ is the maximal operator introduced in \eqref{def_max},
\begin{align}
	\label{def_g_p}
	\begin{split}
		&G_p(y):=\left(\left|\frac{u(y)-(u)_{B_{2R}}}{2R}\right|^{p\gamma}+|\na u(y)|^{p\gamma}+|F(y)|^{p\gamma}\right)\rchi_{B_{2R}}(y)
	\end{split}
\end{align}
and
\begin{align}
	\label{def_g_q}
	\begin{split}
		G_q(y):=[a(y)]^\gamma\left(\left|\frac{u(y)-(u)_{B_{2R}}}{2R}\right|^{q\gamma}+|\na u(y)|^{q\gamma}+|F(y)|^{q\gamma}\right)\rchi_{B_{2R}}(y).
	\end{split}
\end{align}
\item[5.] We denote the lower-level set of the function $G$ defined in \eqref{def_g} as
\begin{equation}
	\label{level_set}
	E\left(\la\right):=\left\{x\in \RR^n\mid G(x)\leqslant \la\right\}\quad (\lambda>0).
\end{equation}
\item[6.] Let $0\leqslant \eta\in W_0^{1,\infty}(B_{2R})$ be a cut-off function with
\begin{equation}
\label{cut_off}
	\begin{array}{lll}
		\eta\equiv 1\text{ in }B_{R}&\text{and}&|\eta|+R|\na \eta|\leqslant c(n)\text{ in }B_{2R},
	\end{array}
\end{equation}
where $B_{2R}$ is the same ball satisfying \eqref{a2}-\eqref{ge_la}. Also, we define another function $v$ by
\begin{align}
	\label{func_v}
	v(x):=\left(u(x)-(u)_{B_{2R}}\right)\eta(x)
\end{align}
\item[7.] Here we point out that the function $v$ in \eqref{func_v} can not be a test function in the equation \eqref{me} due to the lack of integrability assumption of $\na u$. In this regard, we truncate the function $v$ in \eqref{func_v} on the lower-level set of $v$ and $\na v$ and extend the truncated function to be Lipschitz in the upper-level set.  Clearly, the set $E\left(\lambda\right)^c$ in \eqref{level_set} is an open set for every $\lambda>0$ by the lower semi-continuity of the maximal function. In the remaining part of the section, we shall always consider the number $\lambda$ such that 
\begin{align}
	\label{lambda}
	\lambda^{1/p} \geqslant c_{\gamma}\Lambda^{1/p}
\end{align}
with an universal constant $c_{\gamma}$ to be determined via Lemma \ref{p_lip} below and the number $\Lambda$ satisfying \eqref{le_la}-\eqref{ge_la}. As a consequence of Lemma \ref{p_lem} below together with  \eqref{def_ga}-\eqref{pre_de}, we observe that $B_{2R}\cap E(\la^{1/p})\neq \emptyset$. To extend the function $v|_{E(\la^{1/p})}$ to be locally Lipschitz over $\RR^n$, we revisit a classical Whitney covering lemma to the open set $E(\la^{1/p})^c$, which has been widely used for parabolic problems previously, see for instances \cite{B,KL}, and we shall use a version applied in \cite[Section 2.3]{DSSV}. Therefore, there exists a countable family of open balls $\{B_i\}_{i\in\mathbb{N}}$, $B_i\equiv B_{r_i}(x_i)\subset \RR^{n}$ satisfying the following properties:
\begin{description}
	\item[W1\label{w1}] $\bigcup_{i\in \mathbb{N}}\frac{1}{2} B_i= E(\lambda^{1/p})^c$.
	\item[W2\label{w2}] $8B_i\subset E(\la^{1/p})^c$ and $16B_i\cap E(\la^{1/p})\ne \emptyset$ for all $i\in \mathbb{N}$.
	\item[W3\label{w3}] If $B_i\cap B_j\ne \emptyset$ for some $i,j\in \mathbb{N}$, then $\frac{1}{2}r_j\leqslant r_i\leqslant 2r_j$.
	\item[W4] $\frac{1}{4}B_i\cap \frac{1}{4}B_j=\emptyset$ for all $i,j\in\mathbb{N}$ with $i\neq j$.
	\item[W5] Every point $x\in E(\la^{1/p})^c$ belongs to at most $c(n)$ balls of the family $\{4B_i\}_{i\in \mathbb{N}}$.
\end{description}
With the covering $\{B_i\}_{i\in\mathbb{N}}$ fixed as above, there exists a partition of unity $\{\psi_i\}_{i\in\mathbb{N}}\subset C_0^\infty(\RR^{n})$ such that
\begin{description}
	\item[P1] $\displaystyle\rchi_{\frac{1}{2}B_i}\leqslant \psi_i\leqslant \rchi_{\frac{3}{4}B_i}$,
	\item[P2\label{p2}]
	$|\psi_i(x)|+r_i|\na \psi_i(x)|\leqslant c(n)$ for every $x\in \RR^{n}$.
\end{description}
Next we define the set 
\begin{align}
	\label{set_A_i}
	A_i:=\left\{j\in\mathbb{N}\mid \frac{3}{4}B_i\cap\frac{3}{4}B_j\ne\emptyset\right\}
\end{align}
for every $i\in\mathbb{N}$. Therefore, the following properties are also satisfied:
\begin{description}
	\item[P3\label{p3}] $\sum\limits_{j\in A_i}\psi_j=1$ in $B_i$.
	\item[W6\label{w6}] $|B_i\cap B_j|\ge c(n)^{-1}\max\{|B_i|,|B_j|\}$ whenever $j\in A_{i}$.
	\item[W7]: $|\frac{3}{4}B_i\cap\frac{3}{4}B_j|\ge c(n)^{-1}\max\{|B_i|,|B_j|\}$ whenever $j\in A_{i}$.
	\item[W8\label{w8}] the cardinality of $A_i$, $\# A_i$, is uniformly bounded with $\# A_i\leqslant c(n)$.
\end{description}
\item[8.]
Under this Whitney covering as above, we shall consider a truncated function $v_\la $ given by
\begin{align}
	\label{lip_trunc}
	v_\la(x):=v(x)-\sum\limits_{i\in\mathbb{N}}\psi_i(x)(v(x)-v_i),
\end{align}
where
\begin{align}
	\label{v_i}
	v_i:=
	\begin{cases}
		\fint_{\frac{3}{4}B_i}v\ dx &\text{ if }\frac{3}{4}B_i\subset B_{2R},\\
		0&\text{ otherwise.}
	\end{cases}
\end{align}
\end{enumerate}

\subsection{Lipschitz truncation and Caccioppoli inequality}
Under the basic settings and notations in the previous subsection, our next purpose is to construct admissible test functions to the system \eqref{me} by using the techniques of Lipschitz truncation in order to obtain a Caccioppoli type inequality below, see Lemma \ref{caccio}. To go further, first we provide a Poincar\'e type inequality for a very weak solution $u$ to \eqref{me}.

\begin{lemma}
	\label{poinca_lem}
	Under the settings and notations of Subsection \ref{subsec3-1}, there exists a constant $c\equiv c(n,N,p,q,[a]_{0,\alpha})$ such that 
	\begin{align*}
		\begin{split}
			\fint_{B_{2R}}\left[H\left(x,\left|\frac{u-(u)_{B_{2R}}}{2R}\right| \right)\right]^{\delta}\,dx
			\leqslant
			c 
			\fint_{B_{2R}}\left[H(x,|\na u|)\right]^{\delta}\,dx.
		\end{split}
	\end{align*}
\end{lemma}
\begin{proof}
	Let $a(x_m) = \inf\limits_{x\in B_{2R}}a(x)$ for some point $x_{m}\in \overline{B_{2R}}$. Then we have 
	\begin{align}
		\label{poinca:1}
		\begin{split}
			\fint_{B_{2R}}\left[H\left(x,\left|\frac{u-(u)_{B_{2R}}}{2R}\right|\right)\right]^{\delta}\,dx
			&\leqslant
			2
			\fint_{B_{2R}}\left[H\left(x_{m},\left|\frac{u-(u)_{B_{2R}}}{2R}\right|\right)\right]^{\delta}\,dx
			\\&
			\quad
			+
			2\fint_{B_{2R}}|a(x)-a(x_m)|^{\delta}\left|\frac{u-(u)_{B_{2R}}}{2R}\right|^{q\delta}\,dx
			\\&
			=:2\left(I_{1} + I_{2}\right).
		\end{split}
	\end{align}
	Then applying the classical Poincar\'e inequality in \cite[Theorem 1.51]{MZ} and recalling $1<p\gamma<p\delta<p$ and $1<q\gamma<q\delta<q$ by \eqref{def_ga}, we find 
	\begin{align*}
		\begin{split}
			I_{1} 
			&\leqslant 2\fint_{B_{2R}}\left|\frac{u-(u)_{B_{2R}}}{2R}\right|^{p\delta}\,dx + 2 \fint_{B_{2R}}[a(x_m)]^{\delta}\left|\frac{u-(u)_{B_{2R}}}{2R}\right|^{q\delta}\,dx
			\\&
			\leqslant
			c\fint_{B_{2R}}\left[H\left(x_{m},|\na u|\right)\right]^{\delta}\,dx
		\end{split}
	\end{align*}
	for some constant $c\equiv c(n,N,p,q)$. For the estimate of the term $I_{2}$ in \eqref{poinca:1}, using the H\"older continuity of the coefficient function $a(\cdot)$ and classical Sobolev-Poincar\'e inequality in \cite[Corollary 1.64]{MZ} with observing that $\frac{q}{p} < 1+\frac{\alpha}{n} < \frac{n}{n-1}$ to have
	\begin{align*}
		I_{2} 
		&\leqslant 4[a]_{0,\alpha}^{\delta}R^{\alpha\delta}\fint_{B_{2R}}\left(\left|\frac{u-(u)_{B_{2R}}}{2R}\right|^{p\delta}\right)^{\frac{q}{p}}\,dx
		\\&
		\leqslant
		c [a]_{0,\alpha}^{\delta}R^{\alpha\delta}\left(\fint_{B_{2R}}|\na u|^{p\delta}\,dx\right)^{\frac{q}{p}}
		\\&
		\leqslant
		c\left([a]_{0,\alpha}+1\right)R^{\alpha\delta-n\left(\frac{q}{p}-1 \right)}\left( \int_{B_{2R}} |\na u|^{p\delta}\,dx \right)^{\frac{q}{p}-1}
		\fint_{B_{2R}}|\na u|^{p\delta}\,dx
		\\&
		\leqslant
		c \fint_{B_{2R}}|\na u|^{p\delta}\,dx
	\end{align*}
	for some constant $c\equiv c(n,N,p,q,[a]_{0,\alpha})$, where we have also used the assumption \eqref{a2}. Inserting the estimates in the last two displays into \eqref{poinca:1}, we arrive at the conclusion of Lemma \ref{poinca_lem}.
\end{proof}

Next, using the Poincar\'e type inequality of Lemma \ref{poinca_lem}, we derive an estimate for the function $G$ introduced in \eqref{def_g}.

\begin{lemma}\label{p_lem}
	Under the settings and notations of Subsection \ref{subsec3-1}, there exists a constant $c_{\gamma}\equiv c_{\gamma}(n,N,p,q, [a]_{0,\alpha})$ satisfying the following inequality
	\begin{equation*}
		\fint_{B_{2R}}[G(x)]^{\de p}\ dx\leqslant c_{\gamma}\La^{\de}.
	\end{equation*}
	for constants $\delta\in (\delta_0,1)$ and $\Lambda\geqslant 1$ satisfying the conditions \eqref{a2}-\eqref{ge_la}.
\end{lemma}
\begin{proof}
 Clearly, we have 
	\begin{equation*}
		\int_{B_{2R}}[G(x)]^{\de p}\ dx
		\leqslant \int_{\RR^n}\left[M\left(G_p+G_q\right)\right]^\frac{\de}{\ga}(x)\ dx,
	\end{equation*}
	where the functions $G_{p}$ and $G_{q}$ have been defined in \eqref{def_g_p}-\eqref{def_g_q}. Recalling that $1<\frac{\tilde{\ga}}{\ga}<\frac{\de}{\ga}<\frac{1}{\gamma}$ by \eqref{def_ga}-\eqref{pre_de}, we are able to apply Lemma~\ref{p_p} with $s\equiv \frac{\de}{\ga}$ to obtain
	\begin{equation*}
		\int_{B_{2R}}[G(x)]^{\de p}\ dx
		\leqslant c \int_{B_{2R}}\left([G_{p}(x)]^\frac{\delta}{\gamma}+[G_q(x)]^\frac{\delta}{\gamma}\right)\, dx
	\end{equation*}
	for a constant  $c\equiv c(n,p)$. Then, using the Poincar\'e type inequality of Lemma \ref{poinca_lem} and \eqref{le_la}, we have 
	\begin{equation*}
		\fint_{B_{2R}}[G(x)]^{\de p}\ dx
		\leqslant c \fint_{B_{2R}}\left[H(x,|\na u|)+H(x,|F|)\right]^\de\ dx\leqslant c\La^\de
	\end{equation*}
	for some constant $c\equiv c(n,N,p,q, [a]_{0,\alpha})$.
\end{proof}

Next we discuss some important properties of the truncated function $v_\la$ defined in \eqref{lip_trunc}.
\begin{lemma}\label{lip1}
	Under the settings and notations of Subsection \ref{subsec3-1}, there exist constants $c_{p}\equiv c_{p}(n,p)$ and $c_{q}\equiv c_{q}(n,N,p,q,\alpha, [a]_{0,\alpha})$ such that the following inequalities hold true:
	\begin{align}
		\label{lip1:1}
		\begin{split}
			&|v_\la(y)|\leqslant c_{p}R\la^\frac{1}{p}\text{ for all }y\in E(\la^{1/p})^c
			\\&
			\text{and}
			\\&
			[a(y)]^{\frac{1}{q}}|v_{\lambda}(y)| \leqslant c_{q}R \lambda^{\frac{1}{q}}\text{ for all }y\in B_{2R}\cap E(\la^{1/p})^c,
		\end{split}
	\end{align}
	whenever $\lambda\geqslant 1$ is a number satisfying \eqref{lambda}.
\end{lemma}
\begin{proof}
	Let $y\in E(\la^{1/p})^c$ be any fixed point and $\lambda\geqslant 1$ be any fixed number as defined in \eqref{lambda}. Then, the condition \ref{w1} yields that $y\in B_i$ for some $i\in \mathbb{N}$. Using \ref{p3}, we have
	\begin{align*}
		\begin{split}
			|v_\la(y)|\leqslant \sum_{j\in A_i}|v_j|
			&=\sum_{j\in A_i}\fint_{\frac{3}{4}B_j}|u-(u)_{B_{2R}}|\eta\ dx
			\\&\leqslant c(n)\sum_{j\in A_i}\fint_{16 B_j}|u-(u)_{B_{2R}}|\rchi_{B_{2R}}\ dx,
		\end{split}
	\end{align*}
	where the summation is taken over the indices $j\in A_i$ with $\frac{3}{4}B_j\subset B_{2R}$ and the set $A_{i}$ has been defined in \eqref{set_A_i}. The last display together with applying H\"older's inequality and observing that there exists a point $y_{j}\in E(\lambda^{1/p})\cap 16B_{j}$ for every $j\in \mathbb{N}$ via \ref{w2} yield
	\begin{align}
		\label{lip1:2}
		\begin{split}
			|v_{\lambda}(y)| 
			&\leqslant
			c(n)\sum\limits_{j\in A_{i}} \left( \fint_{16B_{j}}|u-(u)_{B_{2R}}|^{p\gamma}\rchi_{B_{2R}}\,dx \right)^{\frac{1}{p\gamma}}
			\\&
			\leqslant
			c(n)R\sum\limits_{j\in A_{i}} \left[ M\left(\left|\frac{u-(u)_{B_{2R}}}{2R} \right|^{p\gamma}\rchi_{B_{2R}} \right)(y_{j}) \right]^{\frac{1}{p\gamma}}
			\\&
			\leqslant
			c(n)R \sum\limits_{j\in A_{i}} G(y_{j})
			\leqslant
			c(n)R\lambda^{\frac{1}{p}}\sum\limits_{j\in A_{i}} 1
			\leqslant
			c(n)R\lambda^{\frac{1}{p}},
		\end{split}
	\end{align}
	where we have also used the definition of the function $G$ in \eqref{def_g} and \ref{w8}. So the first part of \eqref{lip1:1} is proved. Now we turn our attention to the second inequality of \eqref{lip1:1}. In fact, it suffices to show that there exists a constant $c\equiv c(n,N,p,q,\alpha,[a]_{0,\alpha})$ such that
	\begin{align}
		\label{lip1:3}
		\lVert a\rVert_{L^\infty(B_{2R}\cap B_i)}^\frac{1}{q}|v_j|\leqslant cR\la^\frac{1}{q}
	\end{align}
	holds for every $j\in A_i$, where $A_i$ and $v_{j}$ hava been defined in \eqref{set_A_i} and \eqref{v_i}, respectively. We may assume that $\frac{3}{4}B_{j}\subset B_{2R}$. Otherwise, the inequality \eqref{lip1:3} is trivial by the definition $v_{j}$ in \eqref{v_i}. Then, using \ref{w3} and H\"older continuity of the coefficient function $a(\cdot)$, we find 
	\begin{equation*}
		\lVert a\rVert_{L^\infty (B_{2R}\cap B_i)}\leqslant \inf_{x\in  B_{2R}\cap B_j}a(x)+8[a]_{0,\alpha}r_j^\alpha\text{ for all }j\in A_i.
	\end{equation*}
	Therefore, using the last display and recalling the definition of $v_{j}$ in \eqref{v_i}, we have
	\begin{align}
		\label{lip1:4}
		\begin{split}
			\lVert a\rVert_{L^\infty(B_{2R}\cap B_i)}^\frac{1}{q}|v_j(z)|
			\leqslant &c\fint_{\frac{3}{4} B_j}\left[\inf_{x\in  B_{2R}\cap B_j}a(x)\right]^\frac{1}{q}|u-(u)_{B_{2R}}|\ dx
			\\&+cr_j^\frac{\alpha}{q}\fint_{ \frac{3}{4}B_j}|u-(u)_{B_{2R}}|\ dx
			=: c\left( I_{1} + I_{2}\right)
		\end{split}
	\end{align}
	for some constant $c\equiv c(n,q,\alpha,[a]_{0,\alpha})$. Now we shall estimate the resulting terms in the last display. Applying H\"older's inequality together with the observation $1<\ga p < \ga q$, we see
	\begin{align}
		\label{lip1:5}
		\begin{split}
			I_{1}
			&\leqslant cR\left(\fint_{\frac{3}{4} B_j}\left(a(x)\left|\frac{u-(u)_{B_{2R}}}{2R}\right|^{q}\right)^\gamma\ dx\right)^\frac{1}{\ga q}
			\\&\leqslant
			cR \left(\fint_{16B_{j}}\left( a(x)\left|\frac{u-(u)_{B_{2R}}}{2R} \right|^q \right)^{\gamma}\rchi_{B_{2R}}\,dx \right)^{\frac{1}{\gamma q}}
			\\&
			\leqslant
			cR \left[M\left( \left(a(x)\left|\frac{u-(u)_{B_{2R}}}{2R} \right|^{q} \right)^{\gamma}\rchi_{B_{2R}}  \right)(y_{j})\right]^{\frac{1}{\gamma q}}
			\\&
			\leqslant
			cR \left[G(y_{j}) \right]^{\frac{p}{q}}\leqslant
			cR \lambda^{\frac{1}{q}}
		\end{split}
	\end{align}
	for some constant $c\equiv c(n,p,q)$, where $y_{j}\in E(\lambda^{1/p})\cap 16B_{j}$ is a point determined by \ref{w2}. Again applying H\"older's inequality repeatedly together with recalling that 
	\begin{align*}
		1<p\gamma < p\delta<p,\quad \frac{\alpha}{q}-\frac{n}{\delta p}\left(1-\frac{p}{q}\right)<1
		\quad\text{and}\quad
		\frac{1}{\de p}\left(1-\frac{p}{q}\right)<1,
	\end{align*}
 we estimate the term $I_{2}$ in \eqref{lip1:4} as
	\begin{align*}
		\begin{split}
			I_{2} 
			&\leqslant
			2R r_j^\frac{\alpha}{q}\left(\fint_{\frac{3}{4} B_j}\left|\frac{u-(u)_{B_{2R}}}{2R}\right|^{\de p}\ dx\right)^{\frac{1}{\de p}\left(1-\frac{p}{q}\right)}\left(\fint_{\frac{3}{4} B_j}\left|\frac{u-(u)_{B_{2R}}}{2R}\right|^{\ga p}\ dx\right)^{\frac{1}{\gamma p}\frac{p}{q}}
			\\&
			\leqslant
			cR r_j^{\frac{\alpha}{q}-\frac{n}{\de p}\left(1-\frac{p}{q}\right)}\left(\int_{B_{2R}}\left|\frac{u-(u)_{B_{2R}}}{2R}\right|^{\de p}\ dx\right)^{\frac{1}{\de p}\left(1-\frac{p}{q}\right)}
			\\&
			\qquad\times \left(\fint_{16 B_j}\left|\frac{u-(u)_{B_{2R}}}{2R}\right|^{\gamma p}\rchi_{B_{2R}}\ dx\right)^{\frac{1}{\gamma p}\frac{p}{q}}
			\\&
			\leqslant
			cR^{1+\frac{\alpha}{q}-\frac{n}{\de p}\left(1-\frac{p}{q}\right)}\left(\int_{\Omega}\left|\nabla u\right|^{\de p}\ dx\right)^{\frac{1}{\de p}\left(1-\frac{p}{q}\right)}\left(\fint_{16 B_j}\left|\frac{u-(u)_{B_{2R}}}{2R}\right|^{\gamma p}\rchi_{B_{2R}}\ dx\right)^{\frac{1}{\gamma p}\frac{p}{q}}
			\\&
			\leqslant
			cR  \left(\fint_{16 B_j}\left|\frac{u-(u)_{B_{2R}}}{2R}\right|^{\gamma p}\rchi_{B_{2R}}\ dx\right)^{\frac{1}{\gamma p}\frac{p}{q}}
		\end{split}
	\end{align*}
	for some constant $c\equiv c(n,N,p)$, where we have also used a classical Poincar\'e inequality and the fact that $\frac{3}{4} B_{j}\subset B_{2R}$. Using \eqref{lip1:2} in the last display, we conclude
	\begin{align}
		\label{lip1:6}
		I_{2} \leqslant cR\lambda^{\frac{1}{q}}
	\end{align}
	for a constant $c\equiv c(n,N,p)$. Finally, we combine the estimates \eqref{lip1:5}-\eqref{lip1:6} in \eqref{lip1:4} to obtain the second inequality of \eqref{lip1:1}. The proof is completed.
\end{proof}

\begin{lemma}\label{lip2}
	Under the settings and notations of Subsection \ref{subsec3-1}, there exists $c\equiv c(n,N,p)$ such that
	\begin{align}
		\label{lip2:1}
		\fint_{B_i}|v-v_i|\ dx\leqslant c\min\{r_i,R\} \la^\frac{1}{p}
		\text{ for every } i\in\mathbb{N}
	\end{align}
	and
	\begin{align}
		\label{lip2:2}
		|v_i-v_j| \leqslant c\min\{r_i,R\}\lambda^{\frac{1}{p}}
		\text{ for every } j\in A_{i} \text{ and } i\in \mathbb{N},
	\end{align}
	where the set $A_{i}$ has been defined in \eqref{set_A_i}.
\end{lemma}
\begin{proof}
	Let us fix any index $i\in \mathbb{N}$. First, we focus on proving \eqref{lip2:1}. For this, we consider several cases depending on a position of the balls $\frac{3}{4}B_{i}$ and $B_{2R}$. 
	
	\textbf{Case 1: $\frac{3}{4}B_i\subset B_{2R}$ and $B_{i}\subset B_{2R}$.} By the triangle inequality and the Poincar\'e inequality, we see
	\begin{align}
		\label{lip2:3}
		\begin{split}
			\fint_{B_i}|v-v_i|\ dx
			&\leqslant
			c\fint_{B_i}|v-(v)_{B_i}|\ dx
			\\&
			\leqslant
			cr_{i}\fint_{B_{i}}|\na v|\,dx
			\\&
			\leqslant
			cr_i\fint_{B_{i}}\left(|\na u| + \left|\frac{u-(u)_{B_{2R}}}{2R} \right|\right)\,dx
			\\&
			\leqslant
			cr_i\left(\fint_{16B_i}\left(|\na u|^{\ga p}+\left|\frac{u-(u)_{B_{2R}}}{2R}\right|^{\ga p}\right)\rchi_{B_{2R}}\ dx\right)^\frac{1}{\ga p}
			\\&
			\leqslant cr_iG(y_i)\leqslant cr_i\la^\frac{1}{p}
		\end{split}
	\end{align}
	for some constant $c\equiv c(n,N)$, where we have applied H\"older's inequality with the exponent $1<\gamma p$ and the definition of the function $G$ introduced in \eqref{def_g} together with \ref{w2} in such a way that there exists a point $y_i\in E(\lambda^{1/p})^{c}\cap 16B_{i}$.
	
	\textbf{Case 2: $\frac{3}{4}B_{i}\subset B_{2R}$ and $B_{i}\cap B_{2R}^{c}\neq \emptyset$.} In this case, we apply the Poincar\'e inequality of Lemma \ref{boundary_poincare} together with recalling $1<\gamma p$ to have
	
	\begin{align}
		\label{lip2:4}
		\begin{split}
			\fint_{B_i}|v-v_i|\ dx
			&\leqslant c\fint_{4B_i}|v|\ dx\leqslant c r_i\left(\fint_{4B_i}|\na v|^{\gamma p}\ dx\right)^\frac{1}{\gamma p}
			\\&\leqslant c r_i\left(\fint_{16B_i}|\na v|^{\gamma p}\ dx\right)^\frac{1}{\gamma p}
			\\&
			\leqslant
			cr_{i}\left( \fint_{16B_{i}}\left(|\na u|^{\gamma p} + \left|\frac{u-(u)_{B_{2R}}}{2R} \right|^{\gamma p}\right)\rchi_{B_{2R}}\,dx\right)^{\frac{1}{\gamma p}}
		\end{split}
	\end{align}
	for some constant $c\equiv c(n,N,p)$. Arguing similarly as in \eqref{lip2:3}, we arrive at \eqref{lip2:1} in this case. 
	
	\textbf{Case 3: $\frac{3}{4}B_{i}\cap B_{2R}^{c}\neq \emptyset$ and $r_i\leqslant R$.} Again applying Lemma \ref{boundary_poincare} and arguing similarly as in \eqref{lip2:4}, we find 
	\begin{align*}
		\fint_{B_i}|v-v_i|\ dx=\fint_{B_i}|v|\ dx\leqslant cr_i\left(\fint_{4B_i}\left(|\na u|^{\ga p}+\left|\frac{u-(u)_{B_{2R}}}{2R}\right|^{\ga p}\right)\rchi_{B_{2R}} dx\right)^\frac{1}{\ga p}.
	\end{align*}
	Then \eqref{lip2:1} follows from the last display together with the argument in \eqref{lip2:3}.
	
	\textbf{Case 4: $\frac{3}{4}B_{i}\cap B_{2R}^{c}\neq \emptyset$ and $r_i>R$.} Recalling $B_{2R}\cap E(\lambda^{1/p})\neq \emptyset$ and $v\equiv 0$ on $B_{2R}^c$ by the definition of the function $v$ in \eqref{func_v}, we have 
	\begin{align*}
		\begin{split}
			\fint_{B_i}|v-v_i|\ dx=\fint_{B_i}|v|\ dx
			&\leqslant c
			\fint_{B_{2R}}|v|\ dx
			\\&\leqslant cR\fint_{B_{2R}}\left|\frac{u-(u)_{B_{2R}}}{2R}\right|\rchi_{B_{2R}}\,dx\leqslant cR G(y)
		\end{split}
	\end{align*}
	for some constant $c\equiv c(n)$ and a point $y\in B_{2R}\cap E(\lambda^{1/p})$. Taking into account all the cases we have discussed above, the estimate \eqref{lip2:1} is proved. Now we turn our attention to proving \eqref{lip2:2}. For every $i,j\in\mathbb{N}$ with $j\in A_{i}$, we observe that
	\begin{align*}
		\begin{split}
			|v_i-v_j|
			&\leqslant c\fint_{B_i\cap B_j}|v_i-v(x)|\ dx+c\fint_{B_i\cap B_j}|v_j-v(x)|\ dx
			\\&
			\leqslant
			c \fint_{B_i}|v-v_i|\ dx+c \fint_{B_j}|v-v_j|\ dx
		\end{split}
	\end{align*}
	for some constant $c\equiv c(n)$, where we have also applied \ref{w6}. Finally, applying \eqref{lip2:1} on the resulting term of the last display together with recalling $j\in A_{i}$, we arrive at the desired estimate \eqref{lip2:2}. The proof is completed.
\end{proof}

\begin{lemma}\label{lip4}
	Under the settings and notations of Subsection \ref{subsec3-1}, there exist universal constants $c_{lp}\equiv c_{lp}(n,N,p)$ and $c_{lq}\equiv c_{lq}(n,N,p,q,\alpha,[a]_{0,\alpha})$ such that
	\begin{align}
		\label{lip4:1}
		\begin{split}
			&|\na v_\la(y)|\leqslant c_{lp}\la^\frac{1}{p}\text{ for all }y\in E(\lambda^{1/p})^{c}
			\\&
			\quad\text{ and }			
			\\&
			[a(y)]^\frac{1}{q}|\na v_\la(y)|\leqslant c_{lq}\la^\frac{1}{q}\text{ for all }y\in B_{2R}\cap E(\lambda^{1/p})^{c}.
		\end{split}
	\end{align}
\end{lemma}
\begin{proof}
	First, we fix any point $y\in E(\lambda^{1/p})^{c}$. Then, there exists an index $i\in\mathbb{N}$ with $y\in \frac{3}{4}B_{i}$ via \ref{w1}. Therefore, it follows from \ref{p3} that
	\begin{equation*}
		0=\na\left(\sum_{j\in A_i}\psi_j\right)=\sum_{j\in A_i}\na \psi_j\text{ in }B_i.
	\end{equation*}
	Using this one and recalling the definition of the truncated function $v_{\lambda}$ in \eqref{lip_trunc}, we see
	\begin{equation*}
		\na v_\la(y)=\sum_{j\in A_i}v_j \na \psi_j(y)=\sum_{j\in A_i}(v_j -v_i)\na \psi_j(y),
	\end{equation*}
	which implies
	\begin{equation*}
		|\na v_\la(y)|\leqslant \sum_{j\in A_i}|v_j-v_i||\na \psi_j(y)|.
	\end{equation*}
	Then applying Lemma~\ref{lip2} together with \ref{w3}, \ref{w8} and \ref{p2}, we arrive at the first part of \eqref{lip4:1}. Now we shall deal with proving the second estimate of \eqref{lip4:1}. For this, it is enough to show that there exists a constant $c\equiv c(n,N,p,q,\alpha,[a]_{0,\alpha})$ such that
	\begin{align}
		\label{est_lip5}
		I_{0}:=\lVert a\rVert_{L^\infty(B_{2R}\cap B_i)}^\frac{1}{q}\fint_{B_{i}}|v-v_i|\ dx\leqslant c\min\{r_i,R\}\la^{\frac{1}{q}}.
	\end{align}
	In fact, the arguments for proving the above inequality are similar to the ones used in the proof of Lemma \ref{lip2}. We shall divide into several cases depending on a position of the balls $\frac{3}{4}B_{i}$ and $B_{2R}$ as in the proof of Lemma \ref{lip2}.
	
	\textbf{Case 1: $\frac{3}{4}B_i\subset B_{2R}$ and $B_{i}\subset B_{2R}$.} Using triangle inequality and H\"older continuity of the coefficient function $a(\cdot)$, we estimate the term $I_0$ in \eqref{est_lip5} as 
	\begin{align*}
		I_0\le c(n)\lVert a\rVert_{L^\infty(B_i)}^\frac{1}{q}\fint_{B_{i}}|v-(v)_{B_i}|\ dx
	\end{align*} 
	and 
	\begin{align*}
		\lVert a\rVert_{L^\infty(B_i)}\leqslant \inf_{x\in B_i}a(x)+[a]_{0,\alpha}r_i^\alpha.
	\end{align*}
	Then the last two displays together with the Poincar\'e inequality imply
	\begin{align}
		\label{lip4:4}
		\begin{split}
			I_0
			\leqslant &cr_i\fint_{B_i}\left[\inf_{x\in B_i}a(x)\right]^\frac{1}{q}\left(|\na u|+\left|\frac{u-(u)_{B_{2R}}}{2R}\right|\right)\, dx
			\\&+cr_i^{1+\frac{\alpha}{q}}\fint_{B_i}|\na u|+\left|\frac{u-(u)_{B_{2R}}}{2R}\right|\ dx
			=: c(I_{1}+I_{2})
		\end{split}
	\end{align}
	for some constant $c\equiv c(n,N,q,[a]_{0,\alpha})$. Now we estimate the resulting terms $I_1$ and $I_2$ in the last display. For this, using H\"older's inequality and recalling $1<\gamma p < \gamma q$, we have 
	\begin{align*}
		\begin{split}
			I_{1} 
			&\leqslant r_i\left(\fint_{B_i}\left[\inf_{x\in B_i}a(x)\right]^\gamma\left(|\na u|+\left|\frac{u-(u)_{B_{2R}}}{2R}\right|\right)^{q\gamma}\ dx\right)^\frac{1}{q\gamma}
			\\&
			\leqslant cr_i\left(\fint_{16B_i}\left(a(x)\left(|\na u|^q+\left|\frac{u-(u)_{2R}}{2R}\right|^{q}\right)\right)^\gamma\rchi_{B_{2R}}\ dx\right)^{\frac{1}{p\gamma}\frac{p}{q}}
			\\&
			\leqslant
			cr_{i}[G(y_i)]^{\frac{p}{q}} 
			\leqslant
			cr_{i} \lambda^{\frac{1}{q}}
		\end{split}
	\end{align*}
	with some constant $c\equiv c(n,N,p,q,[a]_{0,\alpha})$, where we have also used the fact that there exists a point $y_{i}\in E(\lambda^{1/p})^{c}\cap 16B_{i}$ via \ref{w2}. For the term $I_{2}$ in \eqref{lip4:4}, we apply H\"older's inequality again by observing $1<\gamma p < \gamma q$ in order to obtain 
	\begin{align}
		\label{lip4:6}
		\begin{split}
			I_{2} 
			\leqslant&
			r_i^{1+\frac{\alpha}{q}}\left(\fint_{B_i}\left(|\na u|+\left|\frac{u-(u)_{B_{2R}}}{2R}\right|\right)^{\de p}\ dx\right)^{\frac{1}{\de p}\left(1-\frac{p}{q}\right)}
			\\&\times\left(\fint_{B_i}\left(|\na u|+\left|\frac{u-(u)_{B_{2R}}}{2R}\right|\right)^{\gamma p}\ dx\right)^{\frac{1}{\gamma p}\frac{p}{q}},
		\end{split}
	\end{align}
	where $\delta\in (\delta_0,1)$ is a fixed constant satisfying \eqref{a2}-\eqref{ge_la}. Recalling $\frac{3}{4}r_{i} \leqslant R$ together with applying the Poincar\'e inequality and \eqref{a2}, we observe that 
	\begin{align*}
		\begin{split}
			&r_i^\frac{\alpha}{q}\left(\fint_{B_i}\left(|\na u|+\left|\frac{u-(u)_{B_{2R}}}{2R}\right|\right)^{\de p}\ dx\right)^{\frac{1}{\de p}\left(1-\frac{p}{q}\right)}
			\\&
			\leqslant cR^{\frac{\alpha}{q}-\frac{n}{\de p}\left(1-\frac{p}{q}\right)}\left(\int_{B_{i}}\left(|\na u|+\left|\frac{u-(u)_{B_{2R}}}{2R}\right|\right)^{\de p}\ dx\right)^{\frac{1}{\de p}\left(1-\frac{p}{q}\right)}
			\\&
			\leqslant 
			cR^{\frac{\alpha}{q}-\frac{n}{\de p}\left(1-\frac{p}{q}\right)}\left(\int_{B_{2R}}|\na u|^{\de p}\ dx\right)^{\frac{1}{\de p}\left(1-\frac{p}{q}\right)}\leqslant c
		\end{split}
	\end{align*}
	for some constant $c\equiv c(n,N,p)$. Inserting the content of the last display into \eqref{lip4:6} and using the same argument as in the last part \eqref{lip2:3}, we obtain 
	\begin{align}
		\label{lip4:8}
		\begin{split}
			I_{2} \leqslant c r_{i}\lambda^{\frac{1}{q}}.	
		\end{split}
	\end{align}
	Putting the estimates of \eqref{lip4:6}-\eqref{lip4:8} in \eqref{lip4:4}, we arrive at the validity of \eqref{est_lip5} in the case.
	
	\textbf{Case 2: $\frac{3}{4}B_{i}\subset B_{2R}$ and $B_{i}\cap B_{2R}^{c}\neq \emptyset$.} Using the definition of $v_{i}$ in \eqref{v_i} and the boundary Poincar\'e type inequality of Lemma \ref{boundary_poincare}, $I_0$ in \eqref{est_lip5} can be estimated as follows:
	\begin{align}
		\label{lip4:9}
		I_0
		\leqslant c\lVert a\rVert_{L^\infty(B_{2R}\cap B_i)}^\frac{1}{q}\fint_{B_{i}}|v|\ dx\leqslant c\lVert a\rVert_{L^\infty(B_{2R}\cap 4B_i)}^\frac{1}{q}\left(\fint_{4B_{i}}|\na v|^{\gamma p}\ dx\right)^\frac{1}{\gamma p}
	\end{align}
	with a constant $c\equiv c(n,N,p)$. Recalling that 
	\begin{align*}
		\lVert a\rVert_{L^\infty(B_{2R}\cap 4B_i)} \leqslant \inf\limits_{x\in B_{2R}\cap 4B_{i}}a(x) + 4[a]_{0,\alpha}r_{i}^{\alpha},
	\end{align*}
	we continue to estimate the resulting term in \eqref{lip4:9} with again applying H\"older's inequality as follows
	\begin{align*}
		\begin{split}
			I_{0}
			&\leqslant cr_i\left(\fint_{4B_i}\left[\inf_{x\in B_{2R} \cap 4B_i}a(x)\right]^\gamma\left(|\na u|+\left|\frac{u-(u)_{B_{2R}}}{2R}\right|\right)^{q\gamma}\rchi_{B_{2R}}\ dx\right)^\frac{1}{q\gamma}
			\\&
			\quad 
			+cr_i^{1+\frac{\alpha}{q}}\left(\fint_{4B_i}\left(|\na u|+\left|\frac{u-(u)_{B_{2R}}}{2R}\right|\right)^{p\gamma}\rchi_{B_{2R}}\ dx\right)^\frac{1}{p\gamma}
		\end{split}
	\end{align*}
	for some constant $c\equiv c(n,N,q,\alpha,[a]_{0,\alpha})$. Once we arrive at this stage, the rest of the estimate \eqref{est_lip5} can be argued similarly as shown in \textbf{Case 1}.
	
	\textbf{Case 3: $\frac{3}{4}B_{i}\cap B_{2R}^{c}\neq \emptyset$ and $r_i\leqslant R$.} In this case, we apply the boundary type Poincar\'e inequality of Lemma \ref{boundary_poincare} to obtain
	\begin{align*}
		I_0 =\lVert a\rVert_{L^\infty(B_{2R}\cap B_i)}^\frac{1}{q}\fint_{B_i}|v|\ dx\leqslant c\lVert a\rVert_{L^\infty(B_{2R}\cap 4B_i)}^\frac{1}{q}r_i\left(\fint_{4B_{i}}|\na v|^{\gamma p}\ dx\right)^\frac{1}{\ga p}
	\end{align*}
	for a constant $c\equiv c(n,N,p)$, where we have used the definition of $v_{i}$ in \eqref{v_i}. Then, the remaining can be done in exactly the same way as above.
	
	\textbf{Case 4: $\frac{3}{4}B_{i}\cap B_{2R}^{c}\neq \emptyset$ and $r_i> R$.} By elementary calculations, we have
	\begin{align*}
		\begin{split}
			I_0=\lVert a\rVert_{L^\infty(B_{2R})}^\frac{1}{q}\fint_{B_i}|v|\ dx
			&\leqslant
			c(n)\lVert a\rVert_{L^\infty(B_{2R})}^\frac{1}{q}\fint_{B_{2R}}|v|\ dx
			\\&
			\leqslant c(n)\lVert a\rVert_{L^\infty(B_{2R})}^\frac{1}{q}R\fint_{B_{2R}}\left|\frac{u-(u)_{B_{2R}}}{2R}\right|\ dx
			\\&\leqslant c(n)R\left(\fint_{B_{2R}}\left(\lVert a\rVert_{L^\infty(B_{2R})}\left|\frac{u-(u)_{B_{2R}}}{2R}\right|^q\right)^{\delta}\ dx\right)^\frac{1}{q\delta}.
		\end{split}
	\end{align*}
    Arguing as in the proof of Lemma~\ref{poinca_lem}, we obtain
    \begin{align*}
    	I_0\le cR\left(\fint_{B_{2R}}\left[H(x,|\na u|)\right]^\delta\ dx\right)^\frac{1}{q\delta}
    	\le 
    	cR[G(y)]^\frac{p}{q}\le cR \la^\frac{1}{q}
    \end{align*}
    for some $c\equiv c(n,N,p,q,[a]_{0,\alpha})$, where we have used the fact that there exists $y\in E(\la^{1/p})^c\cap B_{2R}$.
	Finally, taking into account all the cases we considered above, the estimate \eqref{est_lip5} holds true.
\end{proof}

\begin{lemma}\label{p_lip}
	Under the settings and notations of Subsection \ref{subsec3-1}, the function $v_\la$ defined in \eqref{lip_trunc} belongs to $W^{1,\infty}(B_{2R})$.
\end{lemma}

\begin{proof}
	First, we show that there exists a constant $c\equiv c(n,N,p)$ such that the following estimate
	\begin{align}
		\label{p_lip:1}
		\fint_{B_{r}(z)}\left|\frac{v_\la-(v_\la)_{B_r(z)}}{r}\right|\ dx\leqslant c\lambda^{\frac{1}{p}}
	\end{align}
	holds, whenever $B_{r}(z)\subset \RR^{n}$ is a ball. In order to prove the inequality of \eqref{p_lip:1}, we divide several cases depending on a position of the ball $B_{r}(z)$ and the upper level set of the function $G$ in \eqref{def_g}. 
	
	\textbf{Case 1: $B_{r}(z)\subset E(\lambda^{1/p})^{c}$.} Applying the mean value theorem and Lemma \ref{lip4}, we have 
	\begin{align*}
		|v_\la(y)-(v_\la)_{B_r(z)}|\leqslant r\sup_{x\in E(\lambda^{1/p})^c}|\na v_\la(x)|\leqslant cr\la^\frac{1}{p}
	\end{align*}
	for every $y\in B_{r}(z)$ with some constant $c\equiv c(n,N,p)$.
	
	\textbf{Case 2: $B_{r}(z)\cap E(\lambda^{1/p})\neq \emptyset$}. By the elementary inequalities, we find 
	\begin{align}
		\label{p_lip:3}
		\begin{split}
			\fint_{B_{r}(z)}\left|\frac{v_\la-(v_\la)_{B_r(z)}}{r}\right|\ dx
			&\leqslant
			2\fint_{B_{r}(z)}\left|\frac{v_\la-v}{r}\right|\ dx
			\\&
			\quad+2\fint_{B_{r}(z)}\left|\frac{v-(v)_{B_r(z)}}{r}\right|\ dx
			=:2(J_{1}+J_{2}).
		\end{split}
	\end{align}
	Next, we shall estimate the resulting terms in the last display. By the definition of the function $v_{\lambda}$ in \eqref{lip_trunc}, we have 
	\begin{align}
		\label{p_lip:4}
		J_{1} \leqslant \fint_{B_r(z)}\sum_{i\in D}\left|\frac{v-v_i}{r}\right|\psi_i\ dx\leqslant\sum_{i\in D}\frac{1}{|B_r|}\int_{B_r(z)\cap \frac{3}{4}B_i}\left|\frac{v-v_i}{r}\right|\ dx,
	\end{align}
	where 
	\begin{align}
		\label{p_lip:5}
		D= \left\{k\in\mathbb{N} : B_{r}(z)\cap \frac{3}{4}B_{k}\neq \emptyset \right\}.
	\end{align}
 For any fixed ball $B_{i}\equiv B_{r_{i}}(x_i)$ of our covering introduced in Subsection \ref{subsec3-1} and any points $y_1\in B_r(z)\cap \frac{3}{4}B_i$, $y_2\in B_r(z)\cap E(\la^{1/p})$, we observe 
	\begin{align*}
		8r_i\le \dist(x_i,E(\la^{1/p}))\leqslant |x_i-y_1|+|y_1-y_2|\leqslant r_i+2r.
	\end{align*}
	Using the estimate of the last display in \eqref{p_lip:4}, we have 
	\begin{align}
		\label{p_lip:6}
		J_{1} \leqslant \frac{c(n)}{|B_{r}|}\sum\limits_{i\in D}\int_{B_{r}(z)\cap \frac{3}{4}B_{i}} \left|\frac{v-v_i}{r_i}\right|\,dx,
	\end{align}
where the index set $D$ has been defined in \eqref{p_lip:5}. In order to estimate the resulting term in the last display further, we consider subcases depending on a position of the balls of $\frac{3}{4}B_{i} $ and $B_{2R}$ for indices $i\in D $. For indices $i\in D$ such that $\frac{3}{4} B_{i}\subset B_{2R}$ in \eqref{p_lip:6}, recalling that there exists a point $y_{i}\in E(\lambda^{1/p})\cap 16B_{i}$ by \ref{w2} and applying the Poincar\'e inequality, we estimate as
	
	\begin{align}
		\label{p_lip:7}
		\begin{split}
			\int_{B_r(z)\cap \frac{3}{4}B_i}\left|\frac{v-v_i}{r_{i}}\right|\ dx
			&\leqslant 
			c\int_{\frac{3}{4}B_i}|\na v|\ dx
			\\&
			\leqslant c\int_{16B_i}\left(|\na u|+\left|\frac{u-(u)_{B_{2R}}}{2R}\right|\right)\rchi_{B_{2R}}\ dx
			\\&
			\leqslant
			c|16B_i|\left(\fint_{16B_i}\left(|\na u|+\left|\frac{u-(u)_{B_{2R}}}{2R}\right|\right)^{p\gamma}\rchi_{B_{2R}}\ dx\right)^{\frac{1}{p\gamma}}
			\\&
			\leqslant c|16B_i|G(y_i)\leqslant c|16B_i|\la^\frac{1}{p}
		\end{split}
	\end{align}
	for some constant $c\equiv c(n,N)$, where we have also used the definition of the function $G$ in \eqref{def_g}. For indices $i\in D$ such that $\frac{3}{4}B_{i}\cap B_{2R}^{c}\neq \emptyset$ in \eqref{p_lip:6}, we apply Lemma \ref{boundary_poincare} to have 
	\begin{align}
		\label{p_lip:8}
		\int_{B_r(z)\cap \frac{3}{4}B_i}\left|\frac{v-v_i}{r_{i}}\right|\ dx= c\int_{\frac{3}{4}B_i}\frac{|v|}{r_i}\ dx\leqslant c|16B_i|G(y_i)\leqslant c|16B_i|\la^\frac{1}{p}
	\end{align}
	for some constant $c\equiv c(n,N,p)$, where $y_i\in E(\lambda^{1/p})\cap 16B_{i}$ is a point determined via \ref{w2}. Therefore, taking into account both subcases that we have discussed above together with inserting the estimates \eqref{p_lip:7}-\eqref{p_lip:8} into \eqref{p_lip:6}, we have 
	\begin{align*}
		J_{1} \leqslant \frac{c\lambda^{\frac{1}{p}}}{|B_{r}|}\sum\limits_{i\in D} \left|\frac{1}{4}B_{i} \right|.
	\end{align*}
	Recalling the disjointedness of the family $\left\{\frac{1}{4}B_{i}\right\}_{i\in\mathbb{N}}$ and $7r_i\le 2r$ for every $i\in D$ in \eqref{p_lip:5}, we conclude 
	\begin{align}
		\label{p_lip:10}
		J_{1} \leqslant c\lambda^{\frac{1}{p}}.
	\end{align}
	It remains to estimate the term $I_{2}$ in \eqref{p_lip:3}. As in the previous cases, applying the Poincar\'e inequality, we find 
	\begin{align}
		\label{p_lip:11}
		\begin{split}
			J_{2} &\leqslant
			c\fint_{B_{r}(z)}|\na v|\ dx
			\\
			&\leqslant c\fint_{B_{r}(z)}\left(|\na u|+\left|\frac{u-(u)_{B_{2R}}}{2R}\right|\right)\rchi_{B_{2R}}\ dx
			\\&\leqslant cG(y)\leqslant  c\la^\frac{1}{p}
		\end{split}
	\end{align}
	for some constant $c\equiv c(n,N)$ and a point $y\in B_{r}(z)\cap E(\lambda^{1/p})$. Inserting the estimates \eqref{p_lip:10} and \eqref{p_lip:11} into \eqref{p_lip:3}, we arrive at the inequality \eqref{p_lip:1}. Therefore, the H\"older continuity characterization of Campanato and \eqref{p_lip:1} yield the desired result of Lemma \ref{p_lip}. 
\end{proof}

\begin{remark}
Under the settings and notations of Subsection \ref{subsec3-1}, it is immediate to show the following
    \begin{align}
        \label{rmk:3.7:1}
        \lim_{\la\to\infty}\int_{B_R}[H(x,|\na u-\na v_\la|)]^\delta\ dx=0,
    \end{align}
    where the constant $\delta\in (\delta_0,1)$ and the function $v_{\la}$ have been defined in \eqref{a2} and \eqref{lip_trunc}, respectively. In fact, the last display means that any very weak solution $u$ of \eqref{me} under the assumption \eqref{exp} can be approximated locally by Lipschitz truncated functions. For the proof of \eqref{rmk:3.7:1}, recalling \eqref{cut_off} and \eqref{func_v}, it suffices to show the following assertion
    \begin{align}
        \label{rmk:3.7:2}
        \lim_{\la\to\infty}\int_{B_{2R}}[H(x,|\na v-\na v_\la|)]^\delta\ dx=0,
    \end{align}
    where the function $v$ is defined in \eqref{func_v}. For this, by the definition of $v$ and $v_{\la}$, we first observe
    \begin{align}
        \label{rmk:3.7:3}
        \begin{split}
            \int_{B_{2R}}[H(x,|\na v-\na v_\la|)]^\delta\ dx
            &=\int_{E(\la^{1/p})^c}[H(x,|\na v-\na v_\la|)]^\delta\ dx\\
            &\le c\int_{E(\la^{1/p})^c}[H(x,|\na v|)]^\delta\ dx+c\int_{E(\la^{1/p})^c}[H(x,|\na v_\la|)]^\delta\ dx.
        \end{split}
    \end{align}
    for a constant $c\equiv c(p,q)$. Applying Lemma~\ref{poinca_lem} and \eqref{rmk:1}, we find
    \begin{align*}
    \begin{split}
        \int_{B_{2R}}[H(x,|\na v|)]^\delta\ dx
        &\le c\int_{B_{2R}}\left( [H(x,|\na u|)]^\delta+\left[H\left(x,\left|\frac{u-(u)_{B_{2R}}}{R}\right|\right)\right]^\delta\right)\ dx\\
        &\le c\int_{B_{2R}}[H(x,|\na u|)]^\delta\ dx<\infty
    \end{split}
    \end{align*}
    for a constant $c\equiv c(n,N,p,q,[a]_{0,\alpha})$. In turn, the last display implies
    \begin{align*}
        \lim_{\la\to\infty}\int_{E(\la^{1/p})^c}[H(x,|\na v|)]^\delta\ dx=0.
    \end{align*}
    On the other hand, applying again Lemma~\ref{lip4}, we get
    \begin{align*}
    \begin{split}
        \lim_{\la\to\infty}\int_{E(\la^{1/p})^c}[H(x,|\na v_\la|)]^\delta\ dx&\le c\lim_{\la\to\infty}\la^{\delta} |E(\la^{1/p})^c|\\
        &\le c\lim_{\la\to\infty}\int_{E(\la^{1/p})^c}[G(x)]^{p\delta}\,dx=0.
    \end{split}
    \end{align*}
    for a constant $c\equiv c(n,N,p,q,[a]_{0,\alpha})$, where we have used  $E(\la^{1/p})^c=\{x\in\mathbb{R}^n\mid G(x)>\la^{1/p}\}$ and the fact that $G\in L^{\delta p}(B_{2R})$ via Lemma~\ref{p_lem}. Combination of the results from the last two displays in \eqref{rmk:3.7:2} yields \eqref{rmk:3.7:2}. As a direct consequence, the assertion in \eqref{rmk:3.7:1} holds true. 
\end{remark}

\begin{lemma}\label{pre_p_reverse}
	Under the settings and notations of Subsection \ref{subsec3-1}, the following inequality holds true:
	\begin{align}
		\label{pre_p:1}
		\begin{split}
			&\int_{B_{R}\cap E(\la^{1/p})}H(x,|\na u|)\ dx
			\\&\leqslant
			c\int_{B_{2R}\cap E(\la^{1/p})^c}\left(|\na u|^{p-1}\la^\frac{1}{p} +[a(x)]^{\frac{q-1}{q}}|\na u|^{q-1}\la^\frac{1}{q}\right) \ dx
			\\&
			\quad
			+c\int_{B_{2R}\cap E(\la^{1/p})}H\left(x,\left|\frac{u-(u)_{B_{2R}}}{2R}\right|\right)\ dx
			\\&
			\quad 
			+
			c\int_{B_{2R}\cap E(\la^{1/p})^c}\left(|F|^{p-1}\la^\frac{1}{p} +[a(x)]^{\frac{q-1}{q}}|F|^{q-1}\la^\frac{1}{q}\right) \ dx
			\\&
			\quad 
			+c\int_{B_{2R}\cap E(\la^{1/p})}H(x,|F|)\ dx,
		\end{split}
	\end{align}
where $c\equiv c(n,N,p,q,\alpha,\nu,L,[a]_{0,\alpha})$.
\end{lemma}

\begin{proof}
	With the functions $\eta$ defined in \eqref{cut_off} and $v_{\lambda}$ defined in \eqref{lip_trunc}, we take $v_\la\eta^{q}\in W_0^{1,\infty}(B_{2R})$ as a test function to the system \eqref{me} by recalling Lemma \ref{p_lip}. Then we have
	\begin{align}
		\label{pre_p:2}
		\begin{split}
			I_{1}&:=\int_{B_{2R}}\iprod{\mA(x,\na u)}{\na v_\la}\eta^{q}\ dx
			\\&=-q\int_{B_{2R}}\iprod{\mA(x,\na u)}{v_\la\na \eta}\eta^{q-1} dx
			\\&\quad+\int_{B_{2R}}\iprod{|F|^{p-2}F+a(x)|F|^{q-2}F}{\na (v_\la \eta^q)}\ dx=:I_{2}+I_{3}.
		\end{split}
	\end{align}
	Recalling the definition of $v_{\lambda}$ in \eqref{lip_trunc}, we see
	\begin{align}
		\label{pre_p:3}
		\begin{split}
			I_{1} &= \int_{B_{2R}\cap E(\la^{1/p})}\iprod{\mA(x,\na u)}{\na u}\eta^{q+1} \ dx
			\\&
			\quad
			+
			\int_{B_{2R}\cap E(\la^{1/p})}\iprod{\mA(x,\na u)}{(u-(u)_{B_{2R}})\na \eta}\eta^{q}\ dx
			\\&\quad
			+ 
			\int_{B_{2R}\cap E(\la^{1/p})^c}\iprod{\mA(x,\na u)}{\na v_\la}\eta^q\ dx
			=:I_{11}+I_{12}+I_{13}.
		\end{split}
	\end{align}
	Using the structure assumption \eqref{str}, we have 
	\begin{align}
		\label{pre_p:4}
		I_{11}\ge \nu
		\int_{B_{2R}\cap E(\la^{1/p})}H(x,|\na u|)\eta^{q+1}\ dx.
	\end{align}
	Again using \eqref{str} and Young's inequality, we find 
	\begin{align}
		\label{pre_p:5}
		\begin{split}
			I_{12} &\geqslant -c\int_{B_{2R}\cap E(\la^{1/p})}(|\na u|^{p-1}+a(x)|\na u|^{q-1})\left|\frac{u-(u)_{B_{2R}}}{2R}\right|\eta^q\ dx
			\\&
			\geqslant
			-\varepsilon\int_{B_{2R}\cap E(\la^{1/p})}\left(|\na u|^p\eta^{\frac{qp}{p-1}}+a(x)|\na u|^q\eta^{\frac{q^2}{q-1}}\right)\ dx
			\\&
			\quad
			-c_{\varepsilon}\int_{B_{2R}\cap E(\la^{1/p})} H\left(x,\left|\frac{u-(u)_{B_{2R}}}{2R}\right|\right)\ dx
			\\&
			\geqslant
			-\varepsilon\int_{B_{2R}\cap E(\la^{1/p})}H\left(x,|\na u|\right)\eta^{q+1}\ dx
			\\&
			\quad
			-c_{\varepsilon}\int_{B_{2R}\cap E(\la^{1/p})} H\left(x,\left|\frac{u-(u)_{B_{2R}}}{2R}\right|\right)\ dx
		\end{split}
	\end{align}
	with some constant $c_{\varepsilon}\equiv c_{\varepsilon}(p,q,\nu,L,\varepsilon)$, where $\varepsilon>0$ is arbitrarily given. Here, we have also used the fact that 
	\begin{equation*}
		\frac{qp}{p-1}\geqslant \frac{q^2}{q-1}\geqslant q+1.
	\end{equation*}
	For the estimate of $I_{13}$ in \eqref{pre_p:3}, we apply the structure assumption \eqref{str} and then Lemma \ref{lip4} to find  
	\begin{align}
		\label{pre_p:7}
		\begin{split}
			I_{13} &\geqslant 
			-L\int_{B_{2R}\cap E(\la^{1/p})^c}\left(|\na u|^{p-1}+a(x)|\na u|^{q-1}\right)|\na v_\la|\ dx
			\\&
			\geqslant
			-c\int_{B_{2R}\cap E(\la^{1/p})^c}\left(|\na u|^{p-1}\la^\frac{1}{p}+[a(x)]^\frac{q-1}{q}|\na u|^{q-1}\la^\frac{1}{q}\right)\ dx
		\end{split}
	\end{align}
	for some constant $c\equiv c(n,N,p,q,\nu,L)$. Inserting the estimates obtained in \eqref{pre_p:4}, \eqref{pre_p:5} and \eqref{pre_p:7} into \eqref{pre_p:3} and also selecting $\varepsilon\equiv \nu/2$, we have the estimate for $I_{1}$ in \eqref{pre_p:2} as
	\begin{align*}
		\begin{split}
			I_{1} &\geqslant
			\nu/2 \int_{B_{2R}\cap E(\la^{1/p})}H(x,|\na u|)\eta^{q+1}\ dx
			\\&
			\quad
			-
			c\int_{B_{2R}\cap E(\la^{1/p})}H\left(x,\left|\frac{u-(u)_{B_{2R}}}{2R} \right|\right)\,dx
			\\&
			\quad
			-
			c\int_{B_{2R}\cap E(\la^{1/p})^c}\left(|\na u|^{p-1}\la^\frac{1}{p}+[a(x)]^\frac{q-1}{q}|\na u|^{q-1}\la^{\frac{1}{q}}\right)\ dx
		\end{split}
	\end{align*}
	for some constant $c\equiv c(n,N,p,q,\nu,L)$. Now, we turn our attention to the term $I_{2}$ in \eqref{pre_p:2}. Using again \eqref{str}, we see 
	\begin{align}
		\label{pre_p:9}
		\begin{split}
			I_{2} &\leqslant
			c\int_{B_{2R}\cap E(\la^{1/p})}(|\na u|^{p-1}+a(x)|\na u|^{q-1})\left|\frac{u-(u)_{B_{2R}}}{2R}\right|\eta^q\ dx
			\\&
			\quad 
			+
			c\int_{B_{2R}\cap E(\la^{1/p})^c}\left(|\na u|^{p-1}\frac{|v_\la|}{R}+a(x)|\na u|^{q-1}\frac{|v_\la|}{R}\right)\ dx
		\end{split}
	\end{align}
	with a constant $c\equiv c(n,N,\nu,L)$. Then arguing similarly as we have done in \eqref{pre_p:5} for the first term and applying Lemma \ref{lip1} for the second one in \eqref{pre_p:9}, we continue to estimate \eqref{pre_p:9} as
	\begin{align}
		\label{pre_p:10}
		\begin{split}
			I_{2} &\leqslant
			\varepsilon\int_{B_{2R}\cap E(\la^{1/p})}H(x,|\na u|)\eta^{q+1}\ dx
			\\&
			\quad
			+
			c_{\varepsilon} \int_{B_{2R}\cap E(\la^{1/p})}H\left(x,\left|\frac{u-(u)_{B_{2R}}}{2R} \right|\right)\,dx
			\\&
			\quad
			+
			c\int_{B_{2R}\cap E(\la^{1/p})^c}\left(|\na u|^{p-1}\la^{\frac{1}{p}}+[a(x)]^\frac{q-1}{q}|\na u|^{q-1}\la^\frac{1}{q}\right)\ dx
		\end{split}
	\end{align}
	for constants $c_{\varepsilon}\equiv c_{\varepsilon}(n,N,p,q,\nu,L,\varepsilon)$, $c\equiv c(n,N,p,q,\nu,L)$, whenever $\varepsilon>0$. Finally, we shall move onto estimating the term $I_{3}$ in \eqref{pre_p:2}. To this end, we first write it in the following.
	\begin{align}
		\label{pre_p:11}
		\begin{split}
			I_{3} &= \int_{B_{2R}\cap E(\la^{1/p})}\iprod{|F|^{p-2}F+a(x)|F|^{q-2}F}{\na (v_\la \eta^q)}\ dx
			\\&
			\quad 
			+
			\int_{B_{2R}\cap E(\la^{1/p})^c}\iprod{|F|^{p-2}F+a(x)|F|^{q-2}F}{\na (v_\la \eta^q)}\ dx \\&=: I_{31} + I_{32}.
		\end{split}
	\end{align}
	
	Applying Young's inequality, for every $\varepsilon>0$, we have 
	\begin{align}
		\label{pre_p:12}
		\begin{split}
			I_{31}
			&\leqslant
			\varepsilon\int_{B_{2R}\cap E(\la^{1/p})}H(x,|\na (v_\la\eta^q)|)\ dx 
			\\&\quad+ c_{\varepsilon}\int_{B_{2R}\cap E(\la^{1/p})}H(x,|F|)\ dx
			=:
			\varepsilon J_{1} + c_{\varepsilon}J_{2}
		\end{split}
	\end{align}
	with some constant $c_{\varepsilon}\equiv c_{\varepsilon}(n,N,p,q,\nu,\varepsilon)$. We shall deal with the terms appearing in the last display further. Recalling the definition of $v_{\lambda}$ in \eqref{lip_trunc} and using Young's inequality, we find 
	\begin{align*}
		\begin{split}
			J_{1} 
			&\leqslant
			c\int_{B_{2R}\cap E(\la^{1/p})}\left(|\na u|^p\eta^{p(q+1)}+a(x)|\na u|^q\eta^{q(q+1)}\right)\ dx
			\\&
			\quad
			+c\int_{B_{2R}\cap E(\la^{1/p})}H\left(x,\left|\frac{u-(u)_{B_{2R}}}{2R}\right|\right)\ dx
			\\&
			\leqslant
			c\int_{B_{2R}\cap E(\la^{1/p})}H(x,|\na u|)\eta^{q+1}\ dx
			\\&
			\quad
			+c\int_{B_{2R}\cap E(\la^{1/p})}H\left(x,\left|\frac{u-(u)_{B_{2R}}}{2R}\right|\right)\ dx
		\end{split}
	\end{align*}
	with a constant $c\equiv c(n,N,p,q,\nu)$. Therefore, inserting the resulting estimate of the above display into \eqref{pre_p:12} and reabsorbing the terms, we conclude 
	\begin{align}
		\label{pre_p:14}
		\begin{split}
			I_{31} &\leqslant
			\varepsilon \int_{B_{2R}\cap E(\la^{1/p})}H(x,|\na u|)\eta^{q+1}\ dx
			\\&
			\quad
			+ 
			c_{\varepsilon} \int_{B_{2R}\cap E(\la^{1/p})}H\left(x,\left|\frac{u-(u)_{B_{2R}}}{2R}\right|\right)\ dx
			\\&
			\quad
			+
			c_{\varepsilon}\int_{B_{2R}\cap E(\la^{1/p})}H(x,|F|)\ dx
		\end{split}
	\end{align} 
	for a constant $c_{\varepsilon}\equiv c_{\varepsilon}(n,N,p,q,\nu,\varepsilon)$, whenever $\varepsilon>0$ is an arbitrary number. Now we estimate the remaining term $I_{32}$ in \eqref{pre_p:11}. Recalling the definition of $v_{\lambda}$ in \eqref{lip_trunc} again and applying Lemma \ref{lip1} and Lemma \ref{lip4}, we find 
	\begin{align}
		\label{pre_p:15}
		\begin{split}
			I_{32} 
			&\leqslant
			c\int_{B_{2R}\cap E(\la^{1/p})^c}\left[|F|^{p-1}\left(\frac{|v_\la|}{R}+|\na v_\la|\right)\right.
			\\&\qquad
			\left.+[a(x)]^\frac{q-1}{q}|F|^{q-1}\left([a(x)]^\frac{1}{q}\frac{|v_\la|}{R}+[a(x)]^\frac{1}{q}|\na v_\la|\right)\right]\ dx
			\\&
			\leqslant
			c\int_{B_{2R}\cap E(\la^{1/p})^c}\left(|F|^{p-1}\la^\frac{1}{p}+[a(x)]^\frac{q-1}{q}|F|^{q-1}\la^\frac{1}{q}\right) \,dx
		\end{split}
	\end{align}
	with a constant $c\equiv c(n,N,p,q,\nu,L)$. Putting the estimates in \eqref{pre_p:14}-\eqref{pre_p:15} into \eqref{pre_p:11}, for every $\varepsilon>0$, we have 
	\begin{align}
		\label{pre_p:16}
		\begin{split}
			I_{3} 
			&\leqslant
			\varepsilon \int_{B_{2R}\cap E(\la^{1/p})}H(x,|\na u|)\eta^{q+1}\ dx
			\\&
			\quad
			+ 
			c_{\varepsilon} \int_{B_{2R}\cap E(\la^{1/p})}H\left(x,\left|\frac{u-(u)_{B_{2R}}}{2R}\right|\right)\ dx
			\\&
			\quad
			+
			c_{\varepsilon}\int_{B_{2R}\cap E(\la^{1/p})}H(x,|F|)\ dx
			\\&
			\quad
			+
			c\int_{B_{2R}\cap E(\la^{1/p})^c}\left(|F|^{p-1}\la^\frac{1}{p}+[a(x)]^\frac{q-1}{q}|F|^{q-1}\la^\frac{1}{q}\right) \,dx
		\end{split}
	\end{align}
	for some constants $c_{\varepsilon}\equiv c_{\varepsilon}(n,N,p,q,\nu,L,\varepsilon)$ and $c\equiv c(n,N,p,q,\nu,L)$. Finally, collecting all the estimates obtained in \eqref{pre_p:9}, \eqref{pre_p:10} and \eqref{pre_p:16}, inserting them into \eqref{pre_p:2} and then choosing $\varepsilon$ small enough together with reabsorbing resulting terms, we arrive at the desired estimate \eqref{pre_p:1}. The proof is completed.
\end{proof}

Finally, we arrive at the stage of providing a Caccioppoli type inequality under the settings and notations of Subsection \ref{subsec3-1}.
\begin{lemma}
	\label{caccio}
	Under the settings and notations of Subsection \ref{subsec3-1}, there exist constants $\delta_{2}\in (1-1/q,1)$ and $c>0$ depending on $n,N,p,q,\alpha,\nu,L,[a]_{0,\alpha}$ such that if $\delta_{2}\leqslant \delta_0<\delta$, then
	\begin{align}
		\label{cacc:1}
		\begin{split}
			\fint_{B_{R}}[H(x,|\na u|)]^{\de}\ dx
			\leqslant c\fint_{B_{2R}}\left[H\left(x,\left|\frac{u-(u)_{B_{2R}}}{2R}\right|\right)\right]^{\de}\,dx
			+c\fint_{B_{2R}}\left[H(x,|F|)\right]^{\de}\ dx.
		\end{split}
	\end{align}
\end{lemma}
\begin{proof}
	First applying Lemma~\ref{pre_p_reverse}, we have
	\begin{align}
		\label{cacc:2}
		J_{0}\leqslant c_{*}(J_1 + J_{2} + J_{3} + J_{4}),
	\end{align}
	for some constant $c_*\equiv c_*(n,N,p,q,\alpha,\nu,L,[a]_{0,\alpha})$, where
	\begin{align}
		\label{cacc:2_1}
		\begin{split}
		&J_{0}
		:=\int_{c_{\gamma}\La^{1/p}}^\infty \la^{-\kappa}\int_{B_{R}\cap E(\la^{1/p})}H(x,|\na u|)\ dx\ d(\la^\frac{1}{p}),
		\\&J_1:=\int_{c_{\gamma}\La^{1/p}}^\infty \la^{-\kappa}\int_{B_{2R}\cap E(\la^{1/p})^c}\left(|\na u|^{p-1}\la^\frac{1}{p} +[a(x)]^\frac{q-1}{q}|\na u|^{q-1}\la^\frac{1}{q}\right) \ dx\ d (\la^\frac{1}{p}),
        \\&
    	J_2:=\int_{c_{\gamma}\La^{1/p}}^\infty \la^{-\kappa}\int_{B_{2R}\cap E(\la^{1/p})}H\left(x,\left|\frac{u-(u)_{B_{2R}}}{2R}\right|\right)\ dx\ d(\la^\frac{1}{p}),
		\\& J_3:=\int_{c_{\gamma}\La^{1/p}}^\infty \la^{-\kappa}\int_{B_{2R}\cap E(\la^{1/p})^c}\left(|F|^{p-1}\la^\frac{1}{p} +[a(x)]^\frac{q-1}{q}|F|^{q-1}\la^\frac{1}{q}\right) \ dx\ d(\la^\frac{1}{p}),
		\\&J_4:=\int_{c_{\gamma}\La^{1/p}}^\infty \la^{-\kappa}\int_{B_{2R}\cap E(\la^{1/p})}H(x,|F|)\, dx\ d(\la^\frac{1}{p}).
	\end{split}
\end{align}
Here $\kappa:=(1-\de)+1/p$ and the constant $c_{\gamma}\equiv c_{\gamma}(n,N,p,q,[a]_{0,\alpha})$ has been determined in Lemma \ref{p_lem}. In the following, we shall deal with all the terms appearing in \eqref{cacc:2} defined in \eqref{cacc:2_1}. To go further, let us introduce auxiliary notations 
	\begin{align}
		\label{cacc:3}
		\begin{split}
			G_{\Lambda}(x)&:=\max\{c_{\gamma}\La^{1/p},G(x)\}
			\\&\text{and}
			\\U_{R}&:=\{x\in B_{R}\mid H(x,|\na u|) \ge(1-\de)[G_{\Lambda}(x)]^{p}\}.
		\end{split}
	\end{align}
	Recalling \eqref{level_set} and applying Fubini's theorem, we find 
	\begin{align}
		\label{cacc:4}
		\begin{split}
			J_{0} &=  \int_{c_{\gamma}\La^{1/p}}^\infty \la^{-(1-\de)-1/p}\int_{B_{R}}H(x,|\na u|)\rchi_{\left\{G(x)\leqslant \la^{1/p}\right\}}\ dx\ d(\la^\frac{1}{p})
			\\&
			=\int_{B_{R}}H(x,|\na u|)\int_{G_{\Lambda}(x)}^\infty \la^{-(1-\de)-1/p}\ d(\la^\frac{1}{p})\ dx
			\\&
			=\frac{1}{p(1-\de)}\int_{B_{R}}H(x,|\na u|)[G_{\Lambda}(x)]^{-p+\de p}\ dx,
		\end{split}
	\end{align}
	where function $G_{\Lambda}$ has been defined in \eqref{cacc:3}. On the other hand, we observe that 
	\begin{align}
		\label{cacc:5}
		\begin{split}
			\int_{B_{R}}[H(x,|\na u|)]^{\de}\ dx 
			&= 
			\int_{U_{R}}[H(x,|\na u|)]^{\de}\ dx+\int_{B_{R}\setminus U_{R}}[H(x,|\na u|)]^{\de}\ dx
			\\&
			\leqslant
			(1-\de)^{-(1-\de)}\int_{U_{R}}H(x,|\na u|)[G_{\Lambda}(x)]^{-p+\de p}\ dx
			\\&\quad+(1-\de)^{\de}\int_{B_{R}\setminus U_{R}}[G_{\Lambda}(x)]^{\de p}\ dx
			\\&
			\leqslant
			(1-\de)^{-(1-\de)}\int_{B_{R}}H(x,|\na u|)[G_{\Lambda}(x)]^{-p+\de p}\ dx
			\\&\quad+c(1-\de)^{\delta}\La^{\de}|B_{2R}|
		\end{split}
	\end{align}
	for some constant $c\equiv c(n,N,p,q,[a]_{0,\alpha})$, where we have used Lemma \ref{p_lem} and $U_{R}$ is the set introduced in \eqref{cacc:3}. Therefore, combining \eqref{cacc:4}-\eqref{cacc:5}, we conclude
	\begin{align}
		\label{cacc:6}
		\frac{1}{p}\left(\frac{1}{1-\de}\right)^\de\int_{B_{R}}[H(x,|\na u|)]^{\de}\ dx-c\La^{\de}|B_{2R}|\leqslant J_{0}.
	\end{align}
	We now deal with the terms $J_{k}$ for $k\in \{1,2,3,4\}$ appearing in \eqref{cacc:2_1}. In turn, we have 
	\begin{align}
		\label{cacc:7}
		\begin{split}
			J_{1} &= \int_{c_{\gamma}\La^{1/p}}^\infty \la^{-(1-\de)-1/p}\int_{B_{2R}\cap E(\la^{1/p})^c}|\na u|^{p-1}\la^\frac{1}{p}\ dx\ d(\la^\frac{1}{p})
			\\&
			\quad
			+
			\int_{c_{\gamma}\La^{1/p}}^\infty \la^{-(1-\de)-1/p}\int_{B_{2R}\cap E(\la^{1/p})^c}[a(x)]^\frac{q-1}{q}|\na u|^{q-1}\la^\frac{1}{q} \ dx\ d(\la^\frac{1}{p})
			\\&
			=: J_{11} + J_{12}.
		\end{split}
	\end{align}
	Therefore, applying Fubini's theorem and recalling the definition of the level set $E(\la^{1/p})$ in \eqref{level_set}, we see 
	\begin{align*}
		\begin{split}
			J_{11} &= \int_{c_{\gamma}\La^{1/p}}^\infty \la^{-(1-\de)-1/p}\int_{B_{2R}}|\na u|^{p-1}\la^\frac{1}{p}\rchi_{\left\{G(x)\ge\la^{1/p}\right\}}\ dx\ d(\la^\frac{1}{p})
			\\&
			\leqslant
			\int_{B_{2R}} |\na u|^{p-1}\int_{c_{\gamma}\La^{1/p}}^{G(x)}\la^{-(1-\de)}\ d(\la^\frac{1}{p})\,dx
			\\&
			\leqslant
			\frac{1}{1-p(1-\de)}\int_{B_{2R}}|\na u|^{p-1}[G(x)]^{1-p(1-\de)}\ dx
			\\&
			\leqslant
			\frac{1}{1-p(1-\delta)}\int_{B_{2R}}[G(x)]^{p\delta}\,dx,
		\end{split}
	\end{align*}
where we have used the definition of the function $G$ in \eqref{def_g} and the fact that 
	\begin{align*}
		|\na u(x)| \leqslant G(x)
		\quad\text{a.e. in } B_{2R}.
	\end{align*}
	Again using Fubini's theorem, we estimate $J_{12}$ in \eqref{cacc:7} as
	\begin{align}
		\label{cacc:10}
		\begin{split}
			J_{12} &= \int_{B_{2R}}[a(x)]^\frac{q-1}{q}|\na u|^{q-1}\int_{c_{\gamma}\La^{1/p}}^{G(x)} \la^{1/q-(1-\de)-1/p} \  d(\la^\frac{1}{p})\ dx
			\\&
			\leqslant
			\frac{1}{\frac{p}{q}-p(1-\de)}\int_{B_{2R}}([a(x)]^\frac{1}{q}|\na u|)^{q-1}[G(x)]^{\frac{p}{q}-p(1-\de)}\ dx
			\\&
			\leqslant
			\frac{1}{\frac{p}{q}-p(1-\de)}\int_{B_{2R}}[G(x)]^{\delta p}\,dx.
		\end{split}
	\end{align}
	In order to obtain the last inequality, we have used the fact that
	\begin{align*}
		a(x)|\na u(x)|^{q} \leqslant [G(x)]^{p}
		\quad\text{a.e. in }B_{2R}.
	\end{align*}
	Then combining the estimates \eqref{cacc:7}-\eqref{cacc:10}, we conclude
	\begin{align}
		\label{cacc:12}
		J_{1} \leqslant \frac{c}{1-q(1-\delta)}\int_{B_{2R}}[G(x)]^{\delta p}\,dx
	\end{align}
	for some constant $c\equiv c(p,q)$. Now we continue to estimate the next terms in \eqref{cacc:2_1}. In turn, again using Fubini's theorem, we see
	\begin{align}
		\label{cacc:13}
		\begin{split}
			&J_{2}= \int_{B_{2R}}H\left(x,\left|\frac{u-(u)_{B_{2R}}}{2R}\right|\right)\int_{G_{\Lambda}(x)}^\infty \la^{-\kappa}\ d(\la^\frac{1}{p})\ dx
			\\&
			\leqslant
			\int_{B_{2R}}H\left(x,\left|\frac{u-(u)_{B_{2R}}}{2R}\right|\right)\int_{G(x)}^\infty \la^{-\kappa}\ d(\la^\frac{1}{p})\ dx
			\\&
			\leqslant
			\frac{1}{p(1-\de)}\int_{B_{2R}}H\left(x,\left|\frac{u-(u)_{B_{2R}}}{2R}\right|\right)[G(x)]^{-p+\de p}\ dx
			\\&
			\leqslant
			\frac{1}{p(1-\delta)}\int_{B_{2R}}H\left(x,\left|\frac{u-(u)_{B_{2R}}}{2R}\right|\right)^{\delta}\,dx.
		\end{split}
	\end{align}
	The last inequality follows from the fact that
	\begin{align}
		\label{cacc:14}
		\left[H\left(x,\left|\frac{u-(u)_{B_{2R}}}{2R}\right|\right)\right]^{\gamma} \leqslant
		[G(x)]^{p\gamma}
		\quad\text{a.e. in }B_{2R}.
	\end{align}
	Finally, we shall treat the remaining terms $J_{3}$ and $J_{4}$ in  \eqref{cacc:2_1}. Essentially, arguing similarly as we have done in \eqref{cacc:7}-\eqref{cacc:12} for $I_{3}$ and in \eqref{cacc:13}-\eqref{cacc:14} for $I_{4}$, we conclude 
	\begin{align}
		\label{cacc:15}
		\begin{split}
			J_{3} + J_{4} \leqslant
			\frac{c}{1-q(1-\delta)}\int_{B_{2R}}[G(x)]^{\de p}\ dx
			+\frac{1}{p(1-\de)}\int_{B_{2R}}\left[H\left(x,|F|\right)\right]^{\de}\ dx
		\end{split}
	\end{align}
	for some constant $c\equiv c(p,q)$. Inserting the estimates obtained in \eqref{cacc:6}, \eqref{cacc:12}, \eqref{cacc:13} and \eqref{cacc:15} into \eqref{cacc:2} and applying Lemma \ref{p_lem},  we find 
	
	\begin{align*}
		\begin{split}
			&\left(\frac{1}{1-\de}\right)^\de\int_{B_{R}}[H(x,|\na u|)]^{\de}\ dx-c\La^{\de p}|B_{2R}|
			\\&
			\leqslant
			\frac{c}{1-q(1-\delta)}\int_{B_{2R}}[G(x)]^{\de p}\ dx
			+\frac{c}{1-\de}\int_{B_{2R}}\left[H\left(x,\left|\frac{u-(u)_{2R}}{2R}\right|\right)+H(x,|F|)\right]^{\de}\ dx
		\end{split}
	\end{align*} 
	for some constant $c\equiv c(n,N,p,q,[a]_{0,\alpha})$. At this moment, we apply Lemma \ref{p_lem} in order to have 
	\begin{align*}
		\begin{split}
			\int_{B_{R}}\left[H(x,|\na u|)\right]^{\de}\ dx
			&\leqslant
			\frac{c_{*}(1-\delta)^{\delta}}{1-q(1-\delta)}\La^{\de p}|B_{R}|
			\\&\quad+c_{*}(1-\delta)^{\delta-1}\int_{B_{2R}}\left[H\left(x,\left|\frac{u-(u)_{2R}}{2R}\right|\right)\right]^\de\ dx
			\\&\quad+c_{*}(1-\delta)^{\delta-1}\int_{B_{2R}}\left[H\left(x,|F|\right)\right]^{\de}\ dx
		\end{split}
	\end{align*}
	for some constant $c_{*}\equiv c_{*}(n,N,p,q,\alpha,\nu,L,[a]_{0,\alpha})$. Now we take $\de_{2}\equiv \delta_{2}(n,N,p,q,\alpha,\nu,L,[a]_{0,\alpha})$ close enough to $1$ with $\delta_{2}\leqslant \delta_0 < \delta<1$ so that
	\begin{align*}
		\frac{c_{*}(1-\delta)^{\delta}}{1-q(1-\delta)} \leqslant \frac{c_{*}(1-\delta_{2})^{\delta_{2}}}{1-q(1-\delta_{2})} \leqslant \frac{1}{2}.
	\end{align*}
	Taking into account the content of the last display together with recalling \eqref{ge_la} and the fact that there exists a constant $c$ with
	\begin{align*}
		(1-t)^{t-1} = e^{(t-1)\log(1-t)} \leqslant
		e^{|(t-1)\log(1-t)|} < c
		\text{ for all } t\in (0,1),
	\end{align*}
	we find a Caccioppoli type inequality of \eqref{cacc:1}. The proof is completed.
\end{proof}

\begin{lemma}[Reverse H\"older inequality]
	\label{reverse}
	Under the assumptions and conclusions of Lemma \ref{caccio}, there exists a constant $\delta_1\equiv \delta_1(n,N,p,q,\alpha,\nu,L,[a]_{0,\alpha})$ with $\delta_2 < \delta_1 <1$ such that
	if $\de_1\leqslant \de_0<\de<1$, then there holds
	\begin{align}
		\label{reverse:0}
		\begin{split}
			\fint_{B_{R}}\left[H(x,|\na u|)\right]^{\de}\ dx
			\leqslant c\left(\fint_{B_{2R}}\left[H\left(x,|\na u|\right)\right]^{\de_{1}}\ dx\right)^\frac{\de}{\de_{1}}
			+c\fint_{B_{2R}}\left[H\left(x,|F|\right)\right]^{\de}\ dx
		\end{split}
	\end{align}
	for some constant $c\equiv c(n,N,p,q,\alpha,\nu,L,[a]_{0,\alpha})$.
\end{lemma}
\begin{proof}
	Applying the Sobolev-Poincar\'e inequality and observing that $\frac{q}{p} < 1^{*}=\frac{n}{n-1}$, there exists a constant $\delta_{1}\equiv \delta_1(n,N,p,q,\alpha,\nu,L,[a]_{0,\alpha})\in (\delta_{2},1)$ such that 
	\begin{align}
		\label{reverse:1}
		\begin{split}
			&\fint_{B_{2R}}\left(\left|\frac{u-(u)_{B_{2R}}}{2R}\right|^{p}+\inf_{B_{2R}}a(x)\left|\frac{u-(u)_{B_{2R}}}{2R}\right|^{q}\right)\ dx
			\\&
			\leqslant
			c\left[\fint_{B_{2R}}\left(|\na u|^{\de_{1}p}+\left[\inf_{B_{2R}}a(x)\right]^{\de_{1}}|\na u|^{\de_{1} q}\right)\ dx\right]^\frac{1}{\de_{1}}
			\\&\leqslant
			c\left(\fint_{B_{2R}}\left[H\left(x,|\na u|\right)\right]^{\de_{1}}\ dx\right)^\frac{1}{\de_{1}}
		\end{split}
	\end{align}
	for some constant $c\equiv c(n,N,p,q,\alpha,\nu,L,[a]_{0,\alpha})$. In addition, if $\delta_{1}\leqslant \delta_0 < \delta$, then we have
	\begin{align}
		\label{reverse:2}
		\begin{split}
			&R^{\alpha}\fint_{B_{2R}}\left|\frac{u-(u)_{B_{2R}}}{2R}\right|^{q}\ dx 
			\leqslant 
			cR^{\alpha}\left(\fint_{B_{2R}}|\na u|^{\de_{1} p}\ dx\right)^\frac{q}{\de_{1} p}
			\\&
			\leqslant
			cR^{\alpha}\left(\fint_{B_{2R}}|\na u|^{\de p}\ dx\right)^\frac{q-p}{\de p}\left(\fint_{B_{2R}}|\na u|^{\de_{1} p}\ dx\right)^{\frac{1}{\de_{1}}}
			\\&
			\leqslant
			cR^{\alpha-\frac{n(q-p)}{\de p}}\left(\int_{B_{2R}}|\na u|^{\de p}\ dx\right)^\frac{q-p}{\de p}\left(\fint_{B_{2R}}|\na u|^{\de_{1} p}\ dx\right)^\frac{1}{\de_{1}}
			\\&	 	
			\leqslant
			c\left(\fint_{B_{2R}}|\na u|^{\de_{1} p}\ dx\right)^\frac{1}{\de_{1}}
		\end{split}
	\end{align}
	for some constant $c\equiv c(n,N,p,q,\alpha,\nu,L,[a]_{\alpha})$, where we have also used H\"older's inequality and \eqref{a2}. Therefore, if $\delta_1 \leqslant \delta_0< \delta$, then we are able to apply Lemma~\ref{caccio} to obtain 
	\begin{align*}
		\begin{split}
			\fint_{B_{R}}\left[H(x,|\na u|)\right]^{\de}\ dx
			\leqslant
			c\fint_{B_{2R}}\left[H\left(x,\left|\frac{u-(u)_{B_{2R}}}{2R}\right|\right)\right]^{\de}\ dx
			+c\fint_{B_{2R}}\left[H\left(x,|F|\right)\right]^{\de}\ dx
		\end{split}
	\end{align*}
	for some constant $c\equiv c(n,N,p,q,\nu,L,[a]_{0,\alpha})$. Then using the estimates \eqref{reverse:1}, \eqref{reverse:2} and the H\"older continuity of the coefficient function $a(\cdot)$, we see that if $\delta_1\leqslant \delta_0 <\delta$, then we have 
	\begin{align*}
		\begin{split}
			\fint_{B_{2R}}\left[H\left(x,\left|\frac{u-(u)_{B_{2R}}}{2R}\right|\right)\right]^{\de}\ dx
			&
			\leqslant
			\left(\fint_{B_{2R}}H\left(x,\left|\frac{u-(u)_{B_{2R}}}{2R}\right|\right)\ dx\right)^{\delta}
			\\&
			\leqslant
			c\left(\fint_{B_{2R}}\left|\frac{u-(u)_{B_{2R}}}{2R}\right|^{p}+\inf_{B_{2R}}a(x)\left|\frac{u-(u)_{B_{2R}}}{2R}\right|^{q}\ dx\right)^{\delta}
			\\&
			\quad 
			+
			c\left(R^{\alpha}\fint_{B_{2R}}\left|\frac{u-(u)_{B_{2R}}}{2R}\right|^{q}\ dx \right)^{\delta}
			\\&
			\leqslant
			c\left(\fint_{B_{2R}}\left(|\na u|^p+a(x)|\na u|^q\right)^{\de_{1}}\ dx\right)^\frac{\delta}{\delta_{1}}.
		\end{split}
	\end{align*}
	We combine the last two displays to obtain \eqref{reverse:0}. The proof is completed.
\end{proof}

\section{Proof of Theorem~\ref{main}}
Finally, we provide the proof of our main Theorem \ref{main}. It is based on the Gehring lemma \cite{Ge,ME} and the stopping time argument and covering lemma used in \cite{BD,KLe} under the double phase settings. Let $u\in W^{1,1}(\Omega,\mathbb{R}^{N})$ be a very weak solution to the system \eqref{me} with 
\begin{align*}
	\int_{\Omega}\left(|\na u|^{p} + a(x)|\na u|^{q} \right)^{\delta}\,dx < \infty
\end{align*}
for some $\delta\in (\delta_0,1)$, where $\delta_{0}$ is a fixed constant satisfying \eqref{pre_de}-\eqref{pre_a2}, which will be determined at the end of the proof. The remaining part of the proof consists of two main steps.

\textbf{Step 1: Stopping time argument and covering.} Let $B_{2r}\equiv B_{2r}(x_0)\subset \Omega$ be a fixed ball with $2r\leqslant R_0$, where the size of $R_0$ is selected to satisfy
\begin{align}
	\label{pfmain:2}
	R_0^{\frac{\alpha}{q}-\frac{n}{\de p}\left(1-\frac{p}{q}\right)}\left(\int_{\Om}|\na u|^{\de p}\ dx\right)^{\frac{1}{\de p}\left(1-\frac{p}{q}\right)}\leqslant 1.
\end{align}
We select radii $r_1,r_2$ such that $r\leqslant r_1 < r_2 \leqslant 2r$ and consider the following level sets
\begin{align*}
	E(\Lambda,\varrho):= \left\{x\in B_{\varrho}(x_0) : H\left(x, |\na u|\right) \leqslant \Lambda  \right\},
\end{align*}
\begin{align}
	\label{pfmain:4}
	S(\Lambda,\varrho):= \left\{x\in B_{\varrho}(x_0) : H\left(x,|\na u|\right) > \Lambda  \right\}
\end{align}
and 
\begin{align}
	\label{pfmain:5}
	T(\Lambda,\varrho):= \left\{x\in B_{\varrho}(x_0) : H\left(x,|F|\right) > \Lambda  \right\}
\end{align}
for every $r\leqslant \varrho \leqslant 2r$ and $\Lambda > 0$. Also, we define the quantity 
\begin{align*}
	\Phi(B_{\varrho}(y)):= \fint_{B_{\varrho}(y)}\left[H(x,|\na u|)+H(x,|F|)\right]^\de\ dx,
\end{align*}
whenever $B_{\varrho}(y)\subset B_{2r}$ is a ball. Then we observe that 
\begin{align*}
	\lim\limits_{s\to 0^{+}} \Phi(B_{s}(y))>\Lambda^{\delta}
	\text{ for a.e } y\in S(\Lambda^{\delta},\varrho)\text{ with } r\leqslant \varrho \leqslant 2r
\end{align*}
and that if $y\in B_{r_1}$, then we have, for every $\varrho\in \left((r_2-r_1)/20, r_2-r_1 \right)$,
\begin{align*}
	\begin{split}
		\Phi(B_{\varrho}(y)) 
		&\leqslant
		\left(\frac{2r}{\varrho} \right)^{n}\left[1+\fint_{B_{2r}}\left[H(x,|\na u|) + H(x,|F|)\right]^{\delta
		}\,dx \right]
		\\&
		\leqslant
		\left(\frac{40r}{r_2-r_1} \right)^{n}\Lambda_{0}^{\delta},
	\end{split}
\end{align*}
where
\begin{align}
	\label{pfmain:9}
	\Lambda_{0}^{\delta}:= 1+\fint_{B_{2r}}\left[H(x,|\na u|) + H(x,|F|) \right]^{\delta
	}\,dx.
\end{align}
From now on, we shall always consider the values of $\Lambda$ satisfying 
\begin{align*}
	\Lambda^{\delta} > \left(\frac{40r}{r_2-r_1} \right)^{n}\Lambda_{0}^{\delta}.
\end{align*}
Then for almost every $y\in S(\Lambda^{\delta},r_1) $, there exists an \textbf{exit time} radius $\varrho_{y} < (r_2-r_1)/20$ such that 
\begin{align}
	\label{pfmain:11}
	\Phi(B_{\varrho_{y}}(y)) = \Lambda^{\delta}
	\quad\text{and}\quad
	\Phi(B_{\varrho}(y)) < \Lambda^{\delta}
	\text{ for every } \varrho\in (\varrho_{y},r_2-r_1).
\end{align}
Therefore, the family $\{B_{\varrho_{y}}(y)\}$ covers $S(\Lambda^{\delta},r_1)$ up to a negligible set. Observing that $\varrho_{y} \leqslant (r_2-r_1)/20 \leqslant R_0$ for almost every $y\in S(\Lambda^{\delta},r_1)$ and recalling \eqref{pfmain:2}-\eqref{pfmain:11}, we are in a position to apply Lemma \ref{reverse} with $B_{2R}$ replaced by $B_{2\varrho_{y}}(y)$, which implies that there exists an exponent $\delta_{1}\equiv \delta_{1}(n,N,p,q,\alpha,\nu,L,[a]_{0,\alpha})\in(1-1/q,1)$ such that if $\delta_{1}\leqslant \delta_{0}<\delta<1$, then there holds
\begin{align*}
	\begin{split}
		&\fint_{B_{\varrho_y}(y)}\left[H(x,|\na u|) + H(x,|F|)\right]^{\de}\ dx
		\\&
		\leqslant 
		c\left(\fint_{B_{2\varrho_y}(y)}\left[H\left(x,|\na u|\right)\right]^{\de_{1}}\ dx\right)^\frac{\de}{\de_{1}}
		+c\fint_{B_{2\varrho_{y}}(y)}\left[H\left(x,|F|\right)\right]^{\de}\ dx
	\end{split}
\end{align*}
for some constant $c=c(n,N,p,q,\alpha,\nu,L,[a]_{\alpha})$ and almost every $y\in S(\Lambda^{\delta},r_1)$. Then it follows from \eqref{pfmain:11} and H\"older's inequality that
\begin{align}
	\label{pfmain:13}
	\begin{split}
		\fint_{B_{\varrho_y}(y)}\left[ H(x,|\na u|) + H(x,|F|)\right]^{\de}\ dx
		&\leqslant 
		c_{*}\Lambda^{\delta-\delta_{1}}\fint_{B_{2\varrho_y}(y)}\left[H\left(x,|\na u|\right)\right]^{\de_{1}}\ dx
		\\&\quad+c_{*}\fint_{B_{2\varrho_{y}}(y)}\left[H\left(x,|F|\right)\right]^{\de}\ dx
		\\&
		=: c_{*}\left(I_{1} + I_{2} \right)
	\end{split}
\end{align}
for a constant $c_{*}\equiv c_{*}(n,N,p,q,\alpha,\nu,L,[a]_{\alpha})$. Now we shall deal with the terms appearing in the last display. In turn, for any $\theta\in (0,1]$, we have
\begin{align*}
	\begin{split}
		I_{1} 
		&\leqslant
		\Lambda^{\delta-\delta_{1}}\left( \left( \theta\Lambda^{\delta} \right)^{\frac{\delta_{1}}{\delta}} + \frac{1}{|B_{2R_{y}}(y)|}\int_{S(\theta\Lambda^{\delta},r_{2})\cap B_{2\varrho_{y}}(y)}\left[H\left(x,|\na u|\right)\right]^{\delta_{1}}\,dx \right)	
		\\&
		\leqslant
		\theta^{\delta_{1}}\Lambda^{\delta} + 	 \frac{\Lambda^{\delta-\delta_{1}}}{|B_{2R_{y}}(y)|}\int_{S(\theta\Lambda^{\delta},r_{2})\cap B_{2\varrho_{y}}(y)}\left[H\left(x,|\na u|\right)\right]^{\delta_{1}}\,dx
	\end{split}
\end{align*}
and 
\begin{align*}
	I_{2} \leqslant
	\theta\Lambda^{\delta} + \frac{1}{|B_{2\varrho_{y}}(y)|}\int_{T(\theta\Lambda^{\delta},r_2)\cap B_{2\varrho_{y}}(y)}\left[H\left(x,|F|\right)\right]^{\delta}\,dx,
\end{align*}
where the sets $S$ and $T$ have been defined in \eqref{pfmain:4}-\eqref{pfmain:5}. Inserting the estimates obtained in the last two displays into \eqref{pfmain:13}, using \eqref{pfmain:11} and reabsorbing terms, we find 
\begin{align}
	\label{pfmain:16}
	\begin{split}
		&(1-c_{*}\theta^{\delta_{1}}-c_{*}\theta)\fint_{B_{\varrho_y}(y)}\left[ H(x,|\na u|)+H(x,|F|)\right]^{\de}\ dx
		\\&
		\leqslant
		c_{*}\frac{\Lambda^{\delta-\delta_{1}}}{|B_{2R_{y}}(y)|}\int_{S(\theta\Lambda^{\delta},r_{2})\cap B_{2\varrho_{y}}(y)}\left[H(x,|\na u|)\right]^{\delta_{1}}\,dx
		\\&
		\quad
		+
		c_{*}
		\frac{1}{|B_{2\varrho_{y}}(y)|}\int_{T(\theta\Lambda^{\delta},r_2)\cap B_{2\varrho_{y}}(y)}\left[H(x,|F|)\right]^{\delta}\,dx
	\end{split}
\end{align}
for the same constant $c_{*}$ as appeared in \eqref{pfmain:13}. We take $\theta\equiv \theta(n,N,p,q,\alpha,\nu,L,[a]_{0,\alpha})\in (0,1]$ small enough such that 
\begin{align}
	\label{pfmain:16_1}
	1-c_{*}\theta^{\delta_{1}} - c_{*}\theta \geqslant 1/2.
\end{align}
From \eqref{pfmain:16} with the choice of $\theta$ as in the last display, we conclude that
\begin{align*}
	\begin{split}
		\fint_{B_{\varrho_y}(y)}\left[H(x,|\na u|)+H(x,|F|)\right]^{\de}\ dx
		&\leqslant
		c\frac{\Lambda^{\delta-\delta_{1}}}{|B_{2R_{y}}(y)|}\int_{S(\theta\Lambda^{\delta},r_{2})\cap B_{2\varrho_{y}}(y)}\left[H(x,|\na u|)\right]^{\delta_{1}}\,dx
		\\&
		\quad
		+
		c\frac{1}{|B_{2\varrho_{y}}(y)|}\int_{T(\theta\Lambda^{\delta},r_2)\cap B_{2\varrho_{y}}(y)}\left[H(x,|F|)\right]^{\delta}\,dx,
	\end{split}
\end{align*}
for some constant $c\equiv c(n,N,p,q,\alpha,\nu,L,[a]_{0,\alpha})$, whenever $y\in S(\Lambda^{\delta},r_{1})$ in the sense of almost everywhere and $\delta_{1}\equiv\delta_{1}(n,N,p,q,\alpha,\nu,L,[a]_{0,\alpha})\in (1-1/q,1)$ is a constant such that $\delta_{1}\leqslant \delta_{0}$. Now we again use \eqref{pfmain:11} together with the last display to obtain 
\begin{align}
	\label{pfmain:18}
	\begin{split}
		\fint_{10B_{\varrho_{y}}(y)}\left[H\left(x,|\na u|\right)\right]^{\delta}\,dx
		&\leqslant
		\Lambda^{\delta}
		= \fint_{B_{\varrho_{y}}(y)}\left[ H(x,|\na u|)+H(x,|F|)\right]^{\de}\ dx
		\\&
		\leqslant
		c\frac{\Lambda^{\delta-\delta_{1}}}{|B_{2R_{y}}(y)|}\int_{S(\theta\Lambda^{\delta},r_{2})\cap B_{2\varrho_{y}}(y)}\left[H\left(x,|\na u|\right)\right]^{\delta_{1}}\,dx
		\\&
		\quad
		+
		c\frac{1}{|B_{2\varrho_{y}}(y)|}\int_{T(\theta\Lambda^{\delta},r_2)\cap B_{2\varrho_{y}}(y)}\left[H\left(x,|F|\right)\right]^{\delta}\,dx
	\end{split}
\end{align}
for a constant $c\equiv c(n,N,p,q,\alpha,\nu,L,[a]_{0,\alpha})$, where $\theta\in (0,1)$ is the number given in \eqref{pfmain:16_1}. On the other hand, applying Vitali's covering lemma, there exists a countable family of disjoint balls $\{B_{2\varrho_{y_{i}}}(y_i)\}_{i\in \mathbb{N}}\equiv \{2B_{i}\}_{i\in\mathbb{N}}$
such that 
\begin{align*}
	S(\La^\de,r_1)\subset \bigcup_{i\in \mathbb{N}}10B_i\cup \text{ negligible set }\subset B_{r_2}
\end{align*}
and that
\begin{align*}
	\Phi(B_{\varrho_{y_{i}}}(y_{i})) = \Lambda^{\delta}
	\quad\text{and}\quad
	\Phi(B_{\varrho}(y_{i})) < \Lambda^{\delta}
	\text{ for every } \varrho\in (\varrho_{y_{i}},r_2-r_1)
\end{align*}
for every $i\in\mathbb{N}$.
Therefore, it follows from the resulting estimate of \eqref{pfmain:18} that
\begin{align*}
	\begin{split}
		\int_{S(\La^\de,r_1)}\left[H(x,|\na u|)\right]^{\de}\ dx
		&=\sum_{i\in\mathbb{N}}\int_{10 B_i}\left[H(x,|\na u|)\right]^{\de}\ dx
		\\&
		\leqslant
		c\sum_{i\in\mathbb{N}}\La^{\de-\de_{1}}\int_{S(\theta\La^\de,r_2)\cap 2B_i}\left[H\left(x,|\na u|\right)\right]^{\de_{1}}\ dx
		\\&
		\quad
		+
		c\sum_{i\in\mathbb{N}}\int_{T(\theta\La^\de,r_2)\cap 2B_i}\left[H\left(x,|F|\right)\right]^{\de}\ dx
		\\&
		\leqslant
		c\La^{\de-\de_1}\int_{S(\theta\La^\de,r_2)}\left[H\left(x,|\na u|\right)\right]^{\de_1}\ dx
		\\&\quad+
		c\int_{T(\theta\La^\de,r_2)}\left[H\left(x,|F|\right)\right]^{\de}\ dx
	\end{split}
\end{align*}
for some constant $c\equiv c(n,N,p,q,\alpha,\nu,L,[a]_{0,\alpha})$. Taking into account the last display together with the observation that
\begin{equation*}
	\begin{split}
		\int_{S(\theta\La^{\de},r_1)\setminus S(\La^\de,r_1)}\left[H(x,|\na u|)\right]^{\de}\ dx
		\leqslant \La^{\de-\de_{1}}\int_{S(\theta\La^{\de},r_1)\setminus S(\La^\de,r_1)}\left[H(x,|\na u|)\right]^{\de_{1}}\ dx
	\end{split}
\end{equation*}
for any $\La^\de>\left(\frac{40r}{r_2-r_1} \right)^{n}\La_0^\de$, we find 
\begin{align*}
	\begin{split}
		\int_{S(\theta\La^\de,r_1)}\left[H\left(x,|\na u|\right)\right]^{\de}\ dx
		&= 
		\int_{S(\La^\de,r_1)}\left[H\left(x,|\na u|\right)\right]^{\de}\ dx
		\\&
		\quad
		+
		\int_{S(\theta\La^{\de},r_1)\setminus S(\La^\de,r_1)}\left[H\left(x,|\na u|\right)\right]^{\de}\ dx
		\\&
		\leqslant
		c\La^{\de-\de_1}\int_{S(\vartheta\La^\de,r_2)}\left[H\left(x,|\na u|\right)\right]^{\de_1}\ dx
		\\&
		\quad
		+
		c\int_{T(\vartheta\La^\de,r_2)}\left[H\left(x,|F|\right)\right]^{\de}\ dx
	\end{split}
\end{align*}
for some constant $c\equiv c(n,N,p,q,\alpha,\nu,L,[a]_{0,\alpha})$. 
In particular, we have 
\begin{align}
	\label{pfmain:23}
	\begin{split}
		\int_{S(\La^\de,r_1)}\left[H\left(x,|\na u|\right)\right]^{\de}\ dx
		&\leqslant
		c\La^{\de-\de_1}\int_{S(\La^\de,r_2)}\left[H\left(x,|\na u|\right)\right]^{\de_1}\ dx
		\\&
		\quad
		+ c\int_{T(\La^\de,r_2)}\left[H\left(x,|F|\right)\right]^{\de}\ dx
	\end{split}
\end{align}
for some constant $c\equiv c(n,N,p,q,\alpha,\nu,L,[a]_{0,\alpha})$, whenever $\La^\de> \left(\frac{40r}{r_2-r_1} \right)^{n}\La_0^\de$.

\textbf{Step 2: Integration and iteration.} We shall integrate on level sets and use a standard truncation argument to guarantee that the quantities involved are finite. For any fixed constant $k>0$, we denote by
\begin{align}
	\label{pfmain:24}
	\begin{split}
		H_{k}(x,|z|)&:= \min\{H(x,|z|),k\}\quad (x\in\Omega, z\in \RR,\RR^{N} \text{ or } \RR^{nN})
		\\&
		\text{and}
		\\
		S_{k}(\Lambda,\varrho)&:= \left\{x\in B_{\varrho}(x_0) : H_{k}(x,|\na u|)>\Lambda \right\},\quad
		r\leqslant \varrho\leqslant 2r.
	\end{split}
\end{align} 
Clearly, by the definition of the set $S_{k}$ in \eqref{pfmain:24}, we observe that 
\begin{equation}
	\label{pfmain:26}
	S_{k}(\La,\rho)=
	\begin{cases}
		\emptyset&\text{ if }\La>k,\\
		S(\La,\rho)&\text{ if }\La\leqslant k.
	\end{cases}
\end{equation}
Recall that the estimate \eqref{pfmain:23} can be written as
\begin{align*}
	\begin{split}
		\int_{S(\La^\de,r_1)}\left[H\left(x,|\na u|\right)\right]^{\de}\ dx
		&\leqslant
		c\La^{\de-\de_1}\int_{S(\La^\de,r_2)}\left[H\left(x,|\na u|\right)\right]^{\de_1}\ dx
		+ c\int_{T(\La^\de,r_2)}\left[H\left(x,|F|\right)\right]^{\de}\ dx.
	\end{split}
\end{align*}
Then using \eqref{pfmain:26} in the last display, we deduce that 
\begin{align}
	\label{pfmain:28}
	\begin{split}
		&\int_{S_{k}(\La^\de,r_1)}\left[H_{k}\left(x,|\na u|\right)\right]^{\de-\delta_1}\left[H\left(x,|\na u|\right)\right]^{\delta_1}\ dx
		\\&
		\leqslant
		c\La^{\de-\de_1}\int_{S_{k}(\La^\de,r_2)}\left[H\left(x,|\na u|\right)\right]^{\de_1}\ dx
		\\&\quad+ c\int_{T(\La^\de,r_2)}\left[H\left(x,|F|\right)\right]^{\de}\ dx.
	\end{split}
\end{align}
In what follows, we denote 
\begin{align}
	\label{pfmain:28_1}
	\Lambda_{1}^{\delta}:= \left(\frac{40r}{r_2-r_1} \right)^{n}\La_0^\de,
\end{align}
where $\Lambda_{0}$ has been defined in \eqref{pfmain:9}.
Then we multiply \eqref{pfmain:28} by $\Lambda^{-\delta}$ and integrate the resulting inequality with respect to $\Lambda$ for $\Lambda\geqslant \Lambda_{1}$ to discover
\begin{align}
	\label{pfmain:29}
	\begin{split}
		J_{0}
		&:= \int_{\Lambda_{1}}^\infty\La^{-\de}\int_{S_{k}(\La^\de,r_1)}\left[H_{k}\left(x,|\na u|\right)\right]^{\de-\de_1}\left[H\left(x,|\na u|\right)\right]^{\de_1}\ dx\ d\La
		\\&
		\leqslant
		c_{*}\int_{\Lambda_{1}}^\infty\La^{-\de_{1}}\int_{S_{k}(\La^\de,r_2)}\left[H\left(x,|\na u|\right)\right]^{\de_1}\ dx\ d\La
		\\&
		\quad
		+c_{*}\int_{\Lambda_{1}}^\infty\La^{-\de}\int_{T(\La^\de,r_2)}\left[H\left(x,|F|\right)\right]^{\de}\ dx\ d\La
		=:c_{*}(J_{1}+J_{2})
	\end{split}
\end{align}
for some constant $c_{*}\equiv c_{*}(n,N,p,q,\alpha,\nu,L,[a]_{0,\alpha})$. Now we shall estimate the terms appearing in the last display. In turn, applying Fubini's theorem, we find 
\begin{align}
	\label{pfmain:30}
	\begin{split}
		J_{0} &= \int_{\Lambda_{1}}^\infty\La^{-\de}\int_{B_{r_1}(x_0)}\left[H_{k}\left(x,|\na u|\right)\right]^{\de-\de_1}\left[H\left(x,|\na u|\right)\right]^{\de_1}\rchi_{\left\{H_{k}\left(x,|\na u|\right)\geqslant \La\right\}}\ dx\ d\La
		\\&
		=
		\int_{S_{k}(\Lambda_{1}^{\delta},r_1)}\left[H_{k}\left(x,|\na u|\right)\right]^{\de-\de_1}\left[H\left(x,|\na u|\right)\right]^{\de_1}\int_{\Lambda_{1}}^{H_{k}\left(x,|\na u|\right)}\La^{-\de}\ d\La\ dx
		\\&
		=
		\frac{1}{1-\de}\int_{S_{k}(\Lambda_{1}^{\delta},r_1)}\left[H_{k}\left(x,|\na u|\right)\right]^{1-\de_1}\left[H\left(x,|\na u|\right)\right]^{\de_1}\ dx
		\\&\quad-\frac{\Lambda_{1}^{1-\delta}}{1-\de}\int_{S_{k}(\Lambda_{1}^{\delta},r_1)}\left[H_{k}\left(x,|\na u|\right)\right]^{\de-\de_1}\left[H\left(x,|\na u|\right)\right]^{\de_1}\ dx.
	\end{split}
\end{align}
Recalling that $\frac{1-\delta}{\delta}< \frac{1}{\delta_1}$ and using the definition of $\Lambda_{1}$ in \eqref{pfmain:28_1}, we see 
\begin{align*}
	\begin{split}
		&\Lambda_{1}^{1-\de}\int_{S_{k}(\Lambda_{1}^{\delta},r_1)}\left[H_{k}\left(x,|\na u|\right)\right]^{\de-\de_1}\left[H\left(x,|\na u|\right)\right]^{\de_1}\ dx
		\\
		&\leqslant \left(\frac{40r}{r_2-r_1} \right)^\frac{n}{\de_1}\La_0^{1-\de}\int_{B_{2r}(x_0)}\left[H\left(x,|\na u|\right)\right]^\de\ dx
		\\&
		\leqslant
		\left(\frac{40r}{r_2-r_1} \right)^\frac{n}{\de_1}\La_0|B_{2r}(x_0)|.
	\end{split}
\end{align*}
From the estimate of the last display in \eqref{pfmain:30}, we conclude 
\begin{align}
	\label{pfmain:32}
	\begin{split}
		J_{0} &\geqslant
		\frac{1}{1-\de}\int_{S_{k}(\Lambda_{1}^{\delta},r_1)}\left[H_{k}\left(x,|\na u|\right)\right]^{1-\de_1}\left[H\left(x,|\na u|\right)\right]^{\de_1}\ dx
		\\&
		\quad-\frac{1}{1-\de}\left(\frac{40r}{r_2-r_1} \right)^\frac{n}{\de_1}\La_0|B_{2r}|.
	\end{split}
\end{align}
Again applying Fubini's theorem, we see 
\begin{align*}
	\begin{split}
		J_{1} &= \int_{S_{k}(\Lambda_{1}^{\delta},r_{2})}\left[H\left(x,|\na u|\right)\right]^{\de_1}\int_{\Lambda_{1}}^{H_{k}\left(x,|\na u|\right)}\La^{-\de_1}\ d\La\ dx
		\\&
		\leqslant
		\frac{1}{1-\delta_1}\int_{S_{k}(\Lambda_{1}^{\delta},r_{2})}\left[H_{k}\left(x,|\na u|\right)\right]^{1-\delta_1}\left[H\left(x,|\na u|\right)\right]^{\delta_1}\,dx
		\\&\leqslant
		\frac{1}{1-\delta_1}\int_{B_{r_2}(x_0)}\left[H_{k}\left(x,|\na u|\right)\right]^{1-\delta_1}\left[H\left(x,|\na u|\right)\right]^{\delta_1}\,dx
	\end{split}
\end{align*}
and 
\begin{align}
	\label{pfmain:34}
	\begin{split}
		J_{2} \leqslant \frac{1}{1-\delta}\int_{T(\Lambda_{1}^{\delta},r_2)}\left[H\left(x,|F|\right)\right]\,dx \leqslant
		\frac{1}{1-\delta}\int_{B_{r_2}(x_0)}\left[H\left(x,|F|\right)\right]\,dx.
	\end{split}
\end{align}
Combining the resulting estimates of \eqref{pfmain:32}-\eqref{pfmain:34} in \eqref{pfmain:29}, we have 

\begin{align}
	\label{pfmain:35}
	\begin{split}
		&\int_{S_{k}(\Lambda_{1}^{\delta},r_1)}\left[H_{k}\left(x,|\na u|\right)\right]^{1-\de_1}\left[H\left(x,|\na u|\right)\right]^{\de_1}\ dx
		\\&\leqslant  \frac{c(1-\de)}{1-\de_1}\int_{B_{r_2}(x_0)}\left[H_{k}\left(x,|\na u|\right)\right]^{1-\de_1}\left[H\left(x,|\na u|\right)\right]^{\de_1}\ dx
		\\&
		\quad
		+c\int_{B_{r_2}(x_0)}\left[H\left(x,|F|\right)\right]\ dx+\left(\frac{40r}{r_2-r_1} \right)^\frac{n}{\de_1}\La_0|B_{2r}|
	\end{split}
\end{align}
for some constant $c\equiv c(n,N,p,q,\alpha,\nu,L,[a]_{0,\alpha})$. At this moment, recalling the definition of $\Lambda_{1}$ in \eqref{pfmain:28_1} again, we observe
\begin{align*}
	\begin{split}
		&\int_{B_{r_1}(x_0)\setminus S_{k}(\Lambda_{1}^{\delta},r_1)}\left[H_{k}\left(x,|\na u|\right)\right]^{1-\delta_{1}}\left[H\left(x,|\na u|\right)\right]^{\de_1}\ dx
		\\&\leqslant
		\left[\left(\frac{40r}{r_2-r_1} \right)^{\frac{n}{\delta}}\La_0\right]^{1-\de}\int_{S_{k}(\Lambda_{1}^{\delta},r_1)}\left[H_{k}\left(x,|\na u|\right)\right]^{\de-\de_1}\left[H\left(x,|\na u|\right)\right]^{\de_1}\ dx
		\\&
		\leqslant
		\left(\frac{40r}{r_2-r_1} \right)^{\frac{n}{\delta_{1}}}
		\La_0^{1-\de}\int_{B_{2r}(x_0)}\left[H_{k}\left(x,|\na u|\right)\right]^{\de-\de_1}\left[H\left(x,|\na u|\right)\right]^{\de_1}\ dx
		\\&
		\leqslant
		\left(\frac{40r}{r_2-r_1} \right)^{\frac{n}{\delta_{1}}}\La_0|B_{2r}|.
	\end{split}
\end{align*}
Using the resulting inequality of the last display in \eqref{pfmain:35} and recalling that $\delta_0 < \delta$, we have 
\begin{align*}
	\begin{split}
		&\int_{B_{r_1}(x_0)}\left[H_{k}\left(x,|\na u|\right)\right]^{1-\de_1}\left[H\left(x,|\na u|\right)\right]^{\de_1}\ dx
		\\&
		\leqslant  
		\frac{c_{0}(1-\de_0)}{1-\de_1}\int_{B_{r_2}(x_0)}\left[H_{k}\left(x,|\na u|\right)\right]^{1-\de_1}\left[H\left(x,|\na u|\right)\right]^{\de_1}\ dx
		\\&
		\quad
		+c_{0}\int_{B_{r_2}(x_0)}H\left(x,|F|\right)\ dx+c_{0}\left(\frac{40r}{r_2-r_1}\right)^\frac{n}{\de_1}\La_0|B_{2r}|
	\end{split}
\end{align*}
for some constant $c_{0}\equiv c_{0}(n,N,p,q,\alpha,\nu,L,[a]_{0,\alpha})$. Finally, we select $\de_0\in(1-1/q,1)$ so that
\begin{equation*}
	1-1/q<\delta_{1}\leqslant \delta_{0}<1
	\quad\text{and}\quad0<\frac{c_{0}(1-\de_0)}{1-\de_1}\leqslant \frac{1}{2}.
\end{equation*}
In turn, we have 
\begin{align*}
	\begin{split}
		&\int_{B_{r_1}(x_0)}\left[H_{k}\left(x,|\na u|\right)\right]^{1-\de_1}\left[H\left(x,|\na u|\right)\right]^{\de_1}\ dx
		\\&\leqslant  
		\frac{1}{2}\int_{B_{r_2}(x_0)}\left[H_{k}\left(x,|\na u|\right)\right]^{1-\de_1}\left[H\left(x,|\na u|\right)\right]^{\de_1}\ dx
		\\&
		\quad
		+c_{0}\int_{B_{2r}(x_0)}H\left(x,|F|\right)\ dx+c_{0}\left(\frac{40r}{r_2-r_1}\right)^\frac{n}{\de_1}\La_0|B_{2r}|.
	\end{split}
\end{align*}
Once we arrive at this stage, we apply Lemma \ref{iter_lemma} to a bounded function $h : [r,2r]\rightarrow [0,\infty)$ given by
\begin{align*}
	h(t):= \int_{B_{t}(x_0)}\left[H_{k}\left(x,|\na u|\right)\right]^{1-\delta_1}\left[H\left(x,|\na u|\right)\right]^{\delta_1}\,dx
\end{align*}
with the exponents $\gamma_1\equiv 0$ and $\gamma_{2} \equiv \frac{n}{\delta_1}$, to discover
\begin{equation*}
	\int_{B_r(x_0)}\left[H_{k}\left(x,|\na u|\right)\right]^{1-\de_1}\left[H\left(x,|\na u|\right)\right]^{\de_1}\ dx \leqslant c\La_0|B_{2r}|+\int_{B_{2r}(x_0)} H\left(x,|F|\right)\ dx
\end{equation*}
for some constant $c\equiv c(n,N,p,q,\alpha,\nu,L,[a]_{0,\alpha})$. Finally, letting $k\rightarrow \infty$ in the last display and recalling the definition of $\Lambda_{0}$ in \eqref{pfmain:9}, we conclude that
\begin{align*}
	\begin{split}
		\fint_{B_{r}(x_0)}H\left(x,|\na u|\right)\ dx
		&\leqslant
		c\left(\fint_{B_{2r}(x_0)}\left[H\left(x,|\na u|\right)\right]^\de\ dx\right)^\frac{1}{\de}
		+c\left(\fint_{B_{2r}(x_0)}H\left(x,|F|\right)\ dx + 1\right)
	\end{split}
\end{align*}
for some constant $c\equiv c(n,N,p,q,\alpha,\nu,L,[a]_{0,\alpha})$. This completes the proof of Theorem \ref{main}.

\end{document}